\documentclass[12pt]{article}
\usepackage[a4paper, total={6in, 8in}, left=20mm, right=20mm, top=20mm, bottom=20mm]{geometry}
\usepackage{enumitem}
\usepackage{amsmath}
\usepackage{amssymb}
\usepackage{amsthm}
\usepackage{dsfont}
\usepackage{mathrsfs}
\usepackage{esint}

\def\XXint#1#2#3{{\setbox0=\hbox{$#1{#2#3}{\int}$}
\vcenter{\hbox{$#2#3$}}\kern-.5\wd0}}

\newtheorem{lem}{Lemma}

\newtheorem{rem}{Remark}
\newtheorem{cor}{Corollary}
\newtheorem{oq}{Open problem}

\newtheorem{Th}{Theorem}
\newtheorem{Prop}{Proposition}
\newtheorem{Co}{Corollary}

\newtheorem{Dfi}{Definition}
\newtheorem{Rm}{Remark}

\newcommand{\be}{\begin{equation}}
\newcommand{\ee}{\end{equation}}
\newcommand{\R}{\mathbb{R}}

\newcommand{\C}{\mathbb{C}}

\newcommand{\reset}{\setcounter{equation}{0}\setcounter{Th}{0}\setcounter{Prop}{0}\setcounter{Co}{0}
\setcounter{Lm}{0}\setcounter{Rm}{0}}

\def\lf{\left}
\def\rg{\right}

\def\la{\lambda}

\def\ep{\varepsilon}

\def\ov{\overline}
\def\Om{\Omega}
\def\om{\omega}
\def\p{\partial}

\begin{document}
\title{A variational approach to $S^1$-harmonic maps \\
and applications}
\author{Filippo Gaia\footnote{Department of Mathematics, ETH Zentrum,
CH-8093 Z\"urich, Switzerland.} $\, $ and Tristan Rivi\`ere\footnote{Department of Mathematics, ETH Zentrum,
CH-8093 Z\"urich, Switzerland.}}
\maketitle

{\bf Abstract: }{\it We present  a renormalization procedure for the Dirichlet Lagrangian for maps from surfaces with or without boundary into $S^1$, whose finite energy critical points are the $S^1$-harmonic maps with isolated singularities.
We give some applications of this renormalization scheme in two different frameworks. The first application has to do with the renormalization of the Willmore energy for Lagrangian singular immersions into K\"ahler-Einstein surfaces while the second application is dealing with frame energies for surfaces immersions into Euclidian spaces.  }
\medskip

\noindent{\bf Math. Class.} 58E20, 58J05
\section{Introduction}
\subsection{$S^1-$harmonic maps with point singularities}
A map $u\in C^\infty(B^n,S^{m-1})$ is by definition a smooth harmonic map from an $n-$dimensional Euclidian ball into a $m-1$ dimensional sphere in ${\R}^m$ if
\be
\label{I.1}
\Delta u\wedge u=0\quad\mbox{ in }B^n\ .
\ee
where $\Delta$ is denoting the standard negative Laplacian on $B^n$.
Equation (\ref{I.1}) can be interpreted as follows: at any point $x\in B^n$ the Laplacian of the map $u$ is orthogonal to the tangent space of the sphere at $u(x)$:
\be
\label{I.2}
\forall \ x\in B^n\quad \Delta\, u\perp T_{u(x)}S^{m-1}\ .
\ee
This condition generates a non-linear equation known as the {\it harmonic map equation}
\be
\label{I.2-a}
-\Delta u= u |\nabla u|^2\quad\mbox{ in }B^n\ .
\ee
This equation is in fact {\it variational} in the sense that it is the {\it Euler-Lagrange Equation} of the {\it Dirichlet Energy}
\be
\label{I.3}
E(u):=\frac{1}{2}\int_{B^n}|\nabla u|^2\ dx^n\ .
\ee
More precisely smooth solutions to (\ref{I.2}) are smooth critical points of $E$ among maps taking values into $S^{m-1}$ for the following variations
\be
\label{I.4}
\forall \varphi\in C^\infty_0(B^n,{\R}^m)\quad\lf.\frac{d}{dt}\rg|_{t=0}E\lf(\frac{u+t\varphi}{|u+t\,\varphi|}\rg)=0\ .
\ee
Considering harmonic maps which are smooth exclusively is bringing to numerous limitations. For instance, given a smooth map $g_0\ :\ \p B^3\rightarrow S^2$ of topological degree 0
it is still unknown if the following problem has a smooth solution
\begin{align}\label{I.5}
\begin{cases}
-\Delta u=u\lvert \nabla u\rvert^2&\text{ in }B^3\\[5mm]
u=g_0&\text{ on }\partial B^3\ ,
\end{cases}
\end{align}
while it has been proved (see \cite{SU}) by a pure minimization procedure in the Sobolev space $W^{1,2}_{g_0}(B^3,S^2)$ of maps with finite Dirichlet energy and trace equal to $g_0$ that there exists a solution to (\ref{I.5}) with isolated singularities. In fact, there exist boundary data $g_0$ for which any minimizer must have point singularities (see \cite{HL}). These singularities  have a unique {\it tangent cone} of the form
\be
\label{I.5-a}
u_0\, :\, x\in B^3\, \mapsto\frac{x}{|x|}\in S^2\ ,
\ee
modulo the composition by an isometry $R\in O(3)$.
Observe that in one dimension lower the map
\be
\label{I.5-b}
u_0\, :\, z\in D^2\, \mapsto\,\frac{z}{|z|}\in S^1\ .
\ee
is still harmonic away from the origin (in the sense above) but is missing to have finite Dirichlet energy by very little since $|\nabla u_0|\simeq |x|^{-1}$ and only belongs to the {\it Weak Marcinkiewicz Space} or {\it Lorentz Space} $L^{2,\infty}(D^2)$ (see the beginning of section II) but does not belong to $L^2(D^2)$. Nevertheless $u_0$ satisfies a weak version of (\ref{I.1}) in the form
\be
\label{I.6}
\mbox{div}(u\wedge\nabla u)=0\quad\mbox{ in }{\mathcal D}'(D^2)\ .
\ee
The weak solutions to (\ref{I.6}) with point singularities were until now considered as ``semi variational'' in the sense that they solve a weak version of an Euler-Lagrange equation for a Lagrangian, the Dirichlet energy of maps into the circle, which is infinite for these
solutions. In the pioneer work on the subject by F.Bethuel, H.Brezis and F.H\'elein \cite{BBH} these $S^1-$harmonic maps are obtained as weak limit of critical points of the Ginzburg-Landau energy
\be
\label{I.7}
E_\ep(u):=\frac{1}{2}\int_{D^2}|\nabla u|^2+\frac{1}{4\,\ep^2}(1-|u|^2)^2\ dx^2\ .
\ee
One of the drawbacks of these variational formulations is that they are requiring a renormalization procedure due to the asymptotic production of infinite energy and this renormalization can generate tedious and lengthy analysis. Another drawback from these approaches comes from the difficulty to deal with ``natural'' boundary conditions. 

The main purpose of the present work is to remedy to these difficulties and to present a \underbar{direct variational formulation} of singular solutions to (\ref{I.6}). This formulation in particular enable to treat more general boundary conditions than the ones considered in the Ginzburg-Landau theory.
The original motivation for our work is related to the theory of {\it Hamiltonian stationary Lagrangian surfaces} and is explained in section \ref{The hamiltonian stationary condition and Schoen-Wolfson isolated singularities}.

In order to explain our approach we take the simplest framework of maps from ${\C}$ into $S^1$. Let $g\,:\, {\C}\rightarrow S^1\subset {\C}$ such that $\nabla g\in L^{2,\infty}(\C)$. We proceed to the 
Hodge decomposition\footnote{By $g^{-1}\,\nabla g$ we denote the complex multiplication of $g^{-1}$ with $\p_{x_1} g$ and $\p_{x_2}g$ respectively.} in $L^{2,\infty}({\C})$ of $g^{-1}\nabla g$ : there exist two real valued functions  $a_g$ and $b_g$ with $\nabla a_{g}$ and $\nabla b_g$ in $L^{2,\infty}({\C})$ such that
\be
\label{I.8}
g^{-1}\nabla g=i\,\nabla^\perp a_g+i\,\nabla b_g\ ,
\ee
where $\nabla^\perp\cdot:=(-\p_{x_2}\cdot,\p_{x_1}\cdot)$. Using complex coordinates, (\ref{I.8}) becomes also
\be
\label{I.9}
g^{-1}\frac{\p g}{\p \ov{z}}=-\,\frac{\p a_g}{\p \ov{z}}+i\,\frac{\p b_g}{\p \ov{z}}
\ee
which is equivalent\footnote{This computation is performed as a matter of illustration and at this stage we implicitly assume enough regularity on $a_g$ and $b_g$ for $e^{a_g}$ to define a distribution and the chain rule  to hold.} to
\be
\label{I.10}
\frac{\p}{\p\ov{z}}\lf(  e^{a_g-ib_g}\, g  \rg)=0
\ee
For instance, let's assume that $g$ is an $S^1-$harmonic map satisfying (\ref{I.6}) equal to the following product of elementary maps of the form (\ref{I.5-b})  
\be
\label{I.11}
{ g}(z):=\prod_{j=1}^Q \frac{z-p_j}{|z-p_j|}\,\prod_{j=1}^Q \frac{|z-q_j|}{z-q_j}
\ee
on $\mathbb{C}$.
This gives
\[
{ g}^{-1}\nabla{ g}=i\,\sum_{j=1}^Q  \frac{\nabla^\perp|z-p_j|}{|z-p_j|}-i\,\sum_{j=1}^Q   \frac{\nabla^\perp|z-q_j|}{|z-q_j|}\ .
\]
Thus
we can choose $b_{ g}=0$ and
\[
a_{ g}:=\sum_{j=1}^Q\log\frac{|z-p_j|}{|z-q_j|}\ ,
\]
and
\be
\label{I.12}
e^{a_{ g}}{ g}=\prod_{j=1}^Q \frac{z-p_j}{z-q_j}\ .
\ee
is meromorphic. It is then natural to introduce a ``regularization'' of this map by taking its inverse stereographic projection into ${\C}P^1$ :
\begin{align}\label{Equation: Definition of the map ug}
u_g:=\pi^{-1}\lf(e^{a_{ g}}{ g}\rg)
\end{align}
The map we have generated through this procedure is a {\it conformal harmonic map} into $S^2$. We will call it the ``$S^2$ lift'' of the $S^1$ valued map $g$.

We then naturally introduce the following definitions.

\begin{Dfi}
\label{df-I.1}
Let $g\in W^{1,p}(D^2, S^1)$ for some $p>1$.
Let $b_g$ be the unique solution to
\begin{align}\label{I.13}
\begin{cases}
\Delta b_g=\operatorname{div}\left(-ig^{-1}\nabla g\right)&\text{ in }D^2\\[5mm]
b_g=0&\text{ on }\partial D^2.
\end{cases}
\end{align}
%
Let $a_g\in W^{1,p}(D^2,{\R})$ with average 0 on $D^2$ and $u_g$ such that
\begin{align}\label{Equation: Definition of ag}
\nabla^\perp a_g=-ig^{-1}\nabla g-\nabla b_g\ \quad\mbox{ and }\quad\ u_g:=\pi^{-1}\lf(e^{a_{ g}}{ g}\rg)\ ,
\end{align}
where $\pi$ is the stereographic projection from $S^2$ into ${\C}$ sending the north pole to zero. We call $u_g$ the ``$S^2$ lift of $g$''. 
We introduce the ``renormalized Dirichlet Energy'' of $g$ the be the following energy
\[
{\mathcal E}(g):=\frac{1}{4}\int_{D^2}\left(|\nabla u_g|^2+|\nabla b_g|^2\right) dx^2\ .
\]
\hfill $\Box$
\end{Dfi}
\begin{Rm}
\label{rm-I.1}
Observe that the previous definition given for the flat disc extends word by word to the simply connected Riemann surfaces  ${\C}$ and ${\C}P^1$. This is also the case for a general closed Riemann surface $\Sigma$ with or without boundary, modulo the addition of the $L^2$ norms of the harmonic forms involved in the Hodge decomposition of $g^{-1} dg$.
\end{Rm}

The motivation for the denomination ``renormalized Dirichlet Energy'' is justified by the following fact: assume that $g\in W^{1,2}(D^2,S^1)$ is equal to 1 on the boundary, then $g$ admits a lift $\phi\in W^{1,2}_0(D^2,{\R})$
such that $g=e^{i\phi}$. Thus $b_g=\phi$ and 
\[
{\mathcal E}(g)=E(g)=\frac{1}{2}\int_{D^2}|\nabla g|^2\ dx^2.
\]
The advantage of ${\mathcal E}$ over $E$ is that it can be finite still allowing $g$ to have point singularities of the form (\ref{I.5-b}). In particular, if $g$ is defined as in (\ref{I.11})
then the renormalized energy $\mathcal{E}$ (computed as an integral over the whole $\mathbb{C}$) is equal to
\begin{align*}
\mathcal{E}(g)=2\pi Q
\end{align*}
for any choice of points $p_1,...,p_Q, q_1,..., q_Q$ in $\mathbb{C}$.\\
To see this observe that 
since $u_g$ is holomorphic\footnote{Notice that the map $u_g$ can be extended to an holomorphic map on $\mathbb{C}P^1$, therefore its degree is well defined.}
\begin{align*}
\frac{1}{4}\int_{\mathbb{C}}\lvert\nabla u_g\rvert^2\, dx^2=\frac{1}{2}\int_{\mathbb{C}}u_g^*dvol_{S^2}=\frac{1}{2}4\pi\operatorname{deg}(u_g)
\end{align*}
and by considering the preimages of points in a neighbourhood of the north or the south pole we see that
\begin{align*}
\operatorname{deg}(u_g)=Q.
\end{align*}
We recall from \cite{Demengel}, \cite{BMP} and that the obstruction for approximating strongly an arbitrary map $g\in W^{1,1}(D^2, S^1)$ by smooth maps into $S^1$  is given by the distribution
\[
\text{div}(i\,g^{-1}\nabla^\perp g)\ .
\]
More precisely, it is proven in  \cite{BMP} (Theorem 3') that for such a map  $g\in W^{1,1}(D^2, S^1)$
there exists an at most countable family of pairs of points $p_i\in \overline{D^2}$ and integers $d_i$ such that
\[
\text{div}(i\,g^{-1}\nabla^\perp g)=2\pi\sum_{i\in I} d_i\delta_{p_i}\,,
\]
and the convergence has to be understood in the sense that there exists an at most countable family of segments with integer multiplicity such that, denoting $J$ the associated integer rectifiable current,
\[
\partial J=\sum_{i\in I} d_i\delta_{p_i}\quad\mbox{ and }\quad \inf\lf\{M(J)\ ;\ \partial J=\sum_{i\in I} d_i\delta_{p_i}\rg\}\le \int_{D^2}|\nabla g|\ dx^2\ .
\]
Here $M(J)$ denotes the mass of the $1-$current $J$.\\
We have moreover for any sequence $g_k\in C^\infty(D^2,S^1)$ satisfying
\[
g_k\to g\text{ a.e. }
\]
\[
\int_{D^2}|\nabla g|\ dx^2+2\pi\,\inf\lf\{M(J)\ ;\ \partial J=\sum_{i\in I} d_i\delta_{p_i}\rg\}\le \liminf_{k\rightarrow+\infty}\int_{D^2}|\nabla g_k|\ dx^2
\]
and the inequality is optimal for any $g$ (see Theorem 1'
in \cite{BMP}).\\
For a detailed description of maps in $W^{1,1}(\partial D^2, S^1)$ we refer to \cite{SM circle} (especially Chapters 1 and 2).

The mass distribution $$\sum_{i\in I} d_i\delta_{p_i}$$ is also called {\bf topological singular set of $g$} (see \cite{HR}). 

\medskip

In the present work we are interested in the subspace of maps in $W^{1,p}(D^2,S^1)$ such that this mass distribution is discrete.
\begin{Dfi}[Isolated/finite topological singularities]
Let $g\in W^{1,p}(D^2, S^1)$ for some $p>1$. Assume that
\begin{align}\label{Equation: Assumption on gdg}
\operatorname{div}\,(i\,g^{-1}\nabla^\perp g)=2\pi\sum_{i\in I} d_i\,\delta_{p_i}\,,
\end{align} 
where $I$ is an at most countable index set and for any $i\in I$ $p_i\in D^2$ and $d_i\in \mathbb{Z}$. Assume also that the points $p_i$ in $D^2$ are isolated. Any such map will be referred to as a \textbf{$W^{1,p}$ $S^1$-valued map with isolated topological singularities} or \textbf{$S^1$-valued map with discrete topological singular set}.\\
Now assume that $g_0:=g\vert_{\partial D^2}\in W^{1,1}(\partial D^2, S^1)$.
Assume that there exist $Q\in \mathbb{N}$, $p_i\in \overline{D^2}$  $d_i\in \mathbb{Z}$ for any $i\in \{1,..., Q\}$ so that
\begin{align}
\nonumber
i\int_{\partial D^2}g_0^{-1}\partial_{\theta}g_0+2\pi\sum_{\substack{i=1\\ p_i\in D^2}}^{Q} d_i+\pi\sum_{\substack{i=1\\ p_i\in \partial D^2}}^{Q} d_i=0
\end{align}
and assume that for any $p_i\in \partial D^2$ $d_i$ is even\footnote{This assumption will be clarified in Remark \ref{Remark: d_i even for points on the boundary}.}.\\
Assume that for any $\phi\in C^\infty(\overline{D^2})$
\begin{align}
\nonumber
\int_{D^2}i\,g^{-1}\nabla^\perp g \nabla\phi =-i\int_{\partial D^2}\phi\, g_0^{-1}\partial_\theta g_0-2\pi\sum_{\substack{i=1\\ p_i\in D^2}}^{Q}d_i\,\phi(p_i)-\pi\sum_{\substack{i=1\\ p_i\in \partial D^2}}^{Q} d_i\,\phi(p_i).
\end{align}
In this case we refer to the points $p_i\in \overline{D^2}$ as \textbf{topological singularities} of $g$ and we say that \textbf{$g$ has finitely many topological singularities in $\overline{D^2}$}.\hfill$\Box$
\end{Dfi}
The following theorem, which is one of the main results of the present work, gives the sequential weak completeness of $S^1$-valued maps with discrete topological singular set under controlled ``renormalized Dirichlet Energy'' and the sequential weak completeness of $S^1$-valued maps with finitely many topological singularities in $\overline{D^2}$ under controlled ``renormalized Dirichlet Energy''  and  controlled $W^{1,1}$-norm at the boundary.
\begin{Th}
\label{th-I.1}

\begin{enumerate}
\item[a)] Let $(g_k)_{k\in \mathbb{N}}$ be a sequence of maps in $W^{1,(2,\infty)}_{loc}(D^2, S^1)$ uniformly bounded in $W^{1,p}(D^2, S^1)$ for some $p>1$ and such that for any $k\in \mathbb{N}$ $g_k$ has isolated topological singularities in $D^2$.\\
Assume that
\be
\label{I.15}
\limsup_{k\rightarrow +\infty}{\mathcal E}(g_k)<+\infty\ .
\ee
Then there exists a subsequence $g_{k'}$ and a map $g_\infty\in W^{1,(2,\infty)}_{loc}\cap W^{1,p}(D^2,S^1)$ with isolated topological singularities such that 
\be
\label{I.16}
{\mathcal E}(g_\infty)\le\liminf_{k'\rightarrow +\infty}{\mathcal E}(g_k') \quad\mbox{ and }\quad \nabla g_{k'}\rightharpoonup \nabla g_\infty\quad\mbox{ weakly in }L^{(2,\infty)}_{loc}(D^2)\ .
\ee
\hfill $\Box$
\item[b)] Let $g_0\in W^{1,1}\cap H^\frac{1}{2}(\partial D^2, S^1)$.
Let $(g_k)_{k\in \mathbb{N}}$ be a sequence of $S^1$-valued map in $W^{1,p}(D^2, S^1)$ for some $p>1$ (where $p$ might depend on $k$) with finitely many topological singularities in $\overline{D^2}$ and with trace equal to $g_0$ on $\partial D^2$.
Assume that
\begin{align}\label{Equation: Sequence with bounded energy}
\limsup_{k\rightarrow +\infty}\mathcal{E}(g_k)<\infty.
\end{align}
Then there exists a subsequence $g_{k'}$ and a $W^{1,(2,\infty)}$ $S^1$-valued map $g_\infty$ with finitely many topological singularities in $\overline{D^2}$ such that
$$\mathcal E(g_\infty)\leq\liminf_{k'\rightarrow +\infty}{\mathcal E}(g_k) \quad\mbox{ and }\quad \nabla g_{k'}\rightharpoonup \nabla g_\infty\quad\mbox{ weakly in }L^{2,\infty}_{loc}(D^2).$$
\end{enumerate}
\hfill $\Box$
\end{Th}
One important step in the proof of Theorem \ref{th-I.1} consists in the following a-priori estimate on the number of the topological singularities of a function $g\in W^{1,p}(D^2, S^1)$.
\begin{Th}\label{Proposition: A priori estimate on number of singularities}
Let $g_0\in W^{1,1}(\partial D^2, S^1)$. Assume that $g\in W^{1,p}(D^2, S^1)$ (for some $p>1$) is an $S^1$-valued map with finitely many topological singularities in $\overline{D^2}$.
Assume that either no topological singularities lie on $\partial D^2$ or $g_0\in W^{1,1}\cap H^\frac{1}{2}(\partial D^2)$.
Then
\begin{align}
\label{I.18}
\sum_{i=1}^Q \lvert d_i\rvert\leq C\left[\mathcal{E}(g)+\lVert\partial_\theta g_0\rVert_{L^1(\partial D^2)}\right].
\end{align}
for some universal constant $C>0$.
\hfill $\Box$
\end{Th}
\begin{Rm}
\label{rm-I.10}
We believe that Theorem \ref{Proposition: A priori estimate on number of singularities} remains true even if we do not assume a priori that $g$ has finitely many topological singularities in $\overline{D^2}$: the finiteness of the number of singularities should be a consequence of the finiteness of ${\mathcal E}(g)$ combined with the $W^{1,1}$-bound at the boundary.\\
The $W^{1,1}$-bound at the boundary seems necessary 
and it could be that another bound with the same scaling property such as $g\in H^{1/2}(\p D^2, S^1)$ does not imply the finiteness of the number of topological singularities.\footnote{While
the slightly stronger assumption $\Delta^{1/4}g\in  L^{2,1}(\p D^2,S^1)$ (where $L^{2,1}(\p D^2)$ is the Lorentz space pre-dual of the weak $L^2$ space $L^{2,\infty}(\p D^2)$) should imply that $g$ has finitely many topological singularities in $\overline{D^2}$.}\hfill $\Box$
\end{Rm}

The next Theorem is the third main result of the present paper, it says that the
critical points of ${\mathcal E}$ are $S^1-$harmonic maps and vice versa.

\begin{Th}\label{th-I.2}
Let $g\in W^{1,p}(D^2, S^1)$ be as in Definition \ref{df-I.1} .
Assume that
\begin{align}
\nonumber
\mathcal{E}(g)<\infty.
\end{align}
Then $g$ solves the weak $S^1$ harmonic map equation (\ref{I.6}) if and only if $g$ is a critical point of the "renormalized Dirichlet Energy" for smooth variations in the target, that is
\begin{align}
\nonumber
\forall \psi\in C^\infty_c(D^2, \mathbb{R})\quad \frac{d}{dt}\bigg\vert_{t=0}\mathcal{E}(g e^{it\psi})=0.
\end{align}
Moreover, if $g$ has isolated topological singularities $g$ solves the weak $S^1$-harmonic map equation (\ref{I.6}) if and only if its lift $u_g$ is a conformal harmonic map into $S^2$.
\end{Th}
The behaviour of $\mathcal{E}$ under variations of different type is adressed in Remark \ref{Remark: general variations}.\\
Combining the results above and the fact that any map in $W^{1,1}(\partial D^2,S^1)$ admits a finite  ``renormalized Dirichlet Energy'' extension (see Lemma \ref{Lemma: Existence of a competitor})
we obtain the following result.
\begin{Co}
\label{co-I.1}
Let $g_0\in W^{1,1}(\p D^2,S^1)$, then there exists an $S^1$-harmonic map $\displaystyle g_{min}$ minimizing ${\mathcal E}$ among the functions $g\in W^{1,(2,\infty)}(D^2,S^1)$ with finitely many singularities in $\displaystyle \overline{D^2}$ and satisfying $g\vert_{\partial D^2}=g_0$.
\end{Co}
\begin{Rm}
\label{rm-I.2} It is still an open question to know whether or not in Corollary~\ref{co-I.1} the degrees are all equal to $+1$
and whether $Q$ is equal to the topological degree of $g_0$. One could also wonder if one should expect singularities to be located at the boundary or not. While these questions are settled in \cite{BBH} thanks to the
careful analysis of the diverging part of the Ginzburg-Landau energy (i.e. the coefficient in front of $\log \ep^{-1}$), in the present situation there is no such  leading diverging term imposing restrictions on the configuration $(d_i,p_i)$
and these questions are left open at this stage. \hfill $\Box$
\end{Rm}
We conclude this introduction with the following open problem.
\begin{oq}
A more natural trace space than $W^{1,1}(\p D^2,S^1)$ to consider for $S^1$-harmonic map is the trace space $H^{1/2}(\p D^2,S^1)$. In particular it would be interesting to investigate whether any trace in $H^{1/2}(\p D^2,S^1)$ admits 
a finite ``renormalized Dirichlet Energy'' extension - we believe this is the case - and if there exist finite energy extensions (minimal or not) with infinitely many singular points accumulating at the boundary.
This last fact cannot be excluded
a priori.
\end{oq}

The paper is organized as follows.\\
In chapter II we recall the definition of some of the functions spaces we will be using throughout the paper and we fix some notations. We then present some preliminary results about the energy and the functions introduced in Definition \ref{df-I.1}.\\
In section III we give a proof of Theorem \ref{th-I.1} and Theorem \ref{Proposition: A priori estimate on number of singularities}.
In section IV we give a proof of Theorem \ref{th-I.2}.\\
In section V we present two applications of the ideas introduced in this work: the first has to do with the renormalization of the Willmore energy for Lagrangian singular immersions into K\"ahler-Einstein surfaces while the second is dealing with frame energies for surfaces immersions into Euclidian spaces.
\subsubsection*{Aknowledgements}
This work has been supported by the Swiss National Science Foundation (SNF 200020\_192062).

\section{Notation and preliminary results}
\subsection{Notation}
Let $\Omega\subset\mathbb{R}^n$ be a domain.
Recall that a function $f:\Omega\to \mathbb{R}$ is said to belong to the \textbf{weak $L^2$ space} $L^{2,\infty}(\Omega, \mathbb{R})$ if $f$ is measurable and 
$$[ f]_{L^{2,\infty}}:=\sup\{\gamma d_f(\gamma)^\frac{1}{2},\, \gamma>0\}$$
is finite, where
$$d_f(\alpha)=\mathcal{L}^n\left(\left\{x\in \Omega: \lvert f(x)\rvert>\alpha\right\}\right).$$
$[ \cdot]_{L^{2,\infty}}$ is a quasi-norm on $L^{2,\infty}(\Omega, \mathbb{R})$ and $L^{2,\infty}(\Omega, \mathbb{R})$ can be made into a Banach space by introducing a norm $\lVert\cdot\rVert_{L^{2,\infty}}$ equivalent to $[\cdot]_{L^{2,\infty}}$ (as quasi-norm) (see Exercise 1.1.12 in \cite{Gra1}).\\
We also recall the definition of the following space:
$$W^{1, (2,\infty)}(\Omega, \mathbb{R})=\left\{ u\in \mathcal{D}'(\Omega), \, \nabla u\in L^{2,\infty}\right\}.$$
$W^{1, (2,\infty)}(\Omega, \mathbb{R})$ is a Banach space with norm
$$\lVert f\rVert_{W^{1,(2,\infty)}}=\lVert f\rVert_{L^{2,\infty}}+\lVert\nabla f\rVert_{L^{2,\infty}}.$$
Observe that
$$[f]_{W^{1,(2,\infty)}}:=\lVert\nabla f\rVert_{L^{2,\infty}}$$
defines a semi-norm on $W^{1, (2,\infty)}(\Omega, \mathbb{R})$. At times it will be usefull to consider the space $\dot{W}^{1,(2,\infty)}(\Omega)$ obtained as the quotient of $W^{1,(2,\infty)}(\Omega)$ by the constant functions. $\dot{W}^{1,(2,\infty)}(\Omega)$ is again a Banach space and the seminorm
$[\cdot]_{W^{1,(2,\infty)}}$ induces a norm on $\dot{W}^{1,(2,\infty)}(\Omega)$.

In the following we will often consider functions with values in $\mathbb{C}\simeq \mathbb{R}^2$. Sometimes it will be convenient to look at this space as $\mathbb{C}$, while in other occasion as $\mathbb{R}^2$.
To avoid confusion, we will denote the complex multiplication of two elements $\alpha$, $\beta\in \mathbb{C}$ as
$$\alpha\beta,$$
while we will denote their $\mathbb{R}^2$-scalar product as 
$$\alpha\cdot\beta.$$
Moreover, when considering the product of gradients we will use the following notation: if $f,g: \mathbb{R}^2\to \mathbb{R}^n$,
$$<\nabla f, \nabla g>=\sum_{i=1}^2\sum_{j=1}^n \partial_{x_i} f^j\partial_{x_i}g^j.$$
\subsection{Degree of a map between manifolds}
We briefly recall here the notion of degree of a map between smooth manifolds, as we will make large use of it in the present article. For more details see Chapter 7 in \cite{BG}.\\
Let $M$ and $N$ be two oriented, compact, connected smooth $n-$manifolds without boundary.
Let $f:M\to N$ be a smooth map.
For any regular value $y\in N$ of $f$ let
\begin{align*}
\deg(f,y)=\sum_{x\in f^{-1}(y)}\operatorname{sgn}df_x,
\end{align*}
where $\operatorname{sgn}df_x=1$ if $df_x$ is orientation preserving and  $\operatorname{sgn}df_x=-1$ if $df_x$ is orientation reversing.\\
One can show 
that $\deg(f,y)$ does not depend on the choice of $y$, therefore we can define the \textbf{degree of $f$} as
\begin{align*}
\deg(f):=\deg(f,y)
\end{align*}
for any regular value $y\in N$ of the map $f$.\\
The degree of $f$ can also be characterized as follows: $\deg (f)$ is the only integer such that for any smooth $n-$form $\omega$ on $N$
\begin{align*}
\int_Mf^*\omega=\deg (f)\int_N\omega.
\end{align*}
When $M=N=S^1$, the notion of degree of a map from $M$ to $N$ can be extended to continuous maps. In fact given a continuous map $f:S^1\to S^1$ and a continuous parametrization $\phi: [0,1]\to S^1$ of $S^1$ as a closed curve (with $\phi(0)=\phi(1)$)
one can show that there exists a continuous lift $\tilde{f}: [0,1]\to \mathbb{R}$ such that
\begin{align*}
f\circ\phi(x)=e^{i\tilde{f}(x)}\quad\forall x\in S^1.
\end{align*}
Then the \textbf{degree of f} is defined as
\begin{align*}
\deg(f)=\frac{1}{2\pi}\left(\tilde{f}(1)-\tilde{f}(0)\right).
\end{align*}
\subsection{Preliminary results for general $g$}
In this subsection and in the next we collect some preliminary results for functions $g$ as in Definition \ref{df-I.1}. Here we do not make further assumptions on $g$ (in particular we do not assume that $g$ has isolated or finitely many topological singularities). In the next subsection we will focus on functions with finitely many topological singularities.
\begin{lem}\label{Lemma: Explicit computation for the energy}
Let $g\in W^{1,p}( D^2, S^1)$ (for some $p>1$) be as in Definition \ref{df-I.1} and assume that
$$\mathcal{E}(g)<\infty.$$
Let
$$f: \mathbb{R}\to\mathbb{R}, \quad x\mapsto\frac{e^{2x}}{(1+e^{2x})^2}.$$
Then
\begin{align}\label{computations explicit form energy, 1}
\mathcal{E}(g)=\int_{D^2}f(a_g)\left(\lvert \nabla g\rvert^2+\lvert \nabla a_g\rvert^2\right)+\frac{1}{4}\int_{D^2}\lvert \nabla b_g\rvert^2.
\end{align}
Moreover
$$\int_{D^2}f(a_g)\lvert \nabla a_g\rvert^2=\int_{D^2}\lvert \nabla \arctan e^{a_g}\rvert^2$$
and if $b_g=0$ in $D^2$
$$\mathcal{E}(g)=2\int_{D^2}f(a_g)\lvert \nabla a_g\rvert^2.$$
\end{lem}

\begin{proof}
We compute
\begin{align}
\nonumber
\mathcal{E}(g)=\frac{1}{4}\int_{D^2}\left\lvert D\pi^{-1}(e^{a_g}g)\, D(e^{a_g} g)\right\rvert^2+\frac{1}{4}\int_{D^2}\lvert \nabla b_g\rvert^2.
\end{align}
Now
$$D\pi^{-1}(e^{a_g}g)\, D(e^{a_g} g)=\frac{2}{1+\lvert e^{a_g} g\rvert^2}D(e^{a_g} g)=2\frac{e^{a_g}Dg+e^{a_g}g Da_g}{1+e^{2a_g}}.$$
Therefore
\begin{align}
\label{Equation: Computations explicit form energy, in the proof}
\mathcal{E}(g)=\int_{D^2}f(a_g)\left(\lvert \nabla g\rvert^2+\lvert \nabla a_g\rvert^2\right)+\frac{1}{4}\int_{D^2}\lvert \nabla b_g\rvert^2
\end{align}
(here we used the fact that since $g$ takes values in $S^1$, $g \cdot Dg=0$).\\
Moreover let
\begin{align*}
H:\mathbb{R}\to\mathbb{R},\quad x\mapsto \arctan e^x,
\end{align*}
then $H'=f^\frac{1}{2}$ and $H\circ a_g\in W^{1,1}(D^2)$ with
\begin{align*}
\nabla (H\circ a_{g})=\left(f(a_g)\right)^\frac{1}{2}\nabla a_g.
\end{align*}
Therefore
\begin{align*}
\int_{D^2}f(a_g)\lvert \nabla a_g\rvert^2=\int_{D^2}\lvert \nabla (H\circ a_g)\rvert^2=\int_{D^2}\left\lvert \nabla \arctan e^{a_g}\right\rvert^2.
\end{align*}
Finally, if $b_g=0$ in $D^2$ then
$$-ig^{-1}\nabla g=\nabla^\perp a_g,$$
therefore
$$\lvert \nabla g\rvert=\lvert \nabla a_g\rvert$$
and so it follows from (\ref{Equation: Computations explicit form energy, in the proof}) that
$$\mathcal{E}(g)=2\int_{D^2}f(a_g)\lvert \nabla a_g\rvert^2.$$
\end{proof}

\begin{lem}\label{Lemma: Improving the minimizing sequence}
Let $g\in W^{1,p}(D^2, S^1)$ (for some $p>1$) be as in Definition \ref{df-I.1} and assume that
$$\mathcal{E}(g)<\infty.$$
Let
$$\tilde{g}=g e^{-i b_g}.$$
Then $\tilde{g}\in W^{1,p'}(D^2, S^1)$ with $p'=\min\{p,2\}$, $a_{\tilde{g}}=a_g$, $b_{\tilde{g}}=0$
and
\begin{align*}
\mathcal{E}(g)=&2\int_{D^2}f(a_g)\lvert \nabla a_g\rvert^2+\int_{D^2}\left(f(a_g)+\frac{1}{4}\right)\lvert \nabla b_g\rvert^2\\ =&\mathcal{E}(\tilde{g})+\int_{D^2}\left(f(a_g)+\frac{1}{4}\right)\lvert \nabla b_g\rvert^2,
\end{align*}
where $f$ is as in Lemma \ref{Lemma: Explicit computation for the energy}.\\
In particular
$$\mathcal{E}(\tilde{g})\leq \mathcal{E}(g).$$
\end{lem}
\begin{proof}
We compute
$$\nabla\tilde{g}=e^{-i b_g}\nabla g-ige^{-i b_g}\nabla b_g.$$
Therefore $\tilde{g}\in W^{1,p'}(D^2, S^1)$, where $p'=\min\{p,2\}$, and
$$\tilde{g}^{-1}\nabla \tilde{g}=g^{-1}\nabla g-i\nabla b_g=i\nabla^\perp a_g.$$
Thus
$$a_{\tilde{g}}=a_g\text{ and }b_{\tilde{g}}=0.$$
By Lemma \ref{Lemma: Explicit computation for the energy} there holds
\begin{align}\label{Equation: Explicit form of Energy in terms of a,b; 1}
\mathcal{E}(\tilde{g})=2\int_{D^2}f(a_g)\lvert \nabla a_g\rvert^2
\end{align}
and
\begin{align}\label{Equation: Explicit form of Energy in terms of a,b; 2}
\mathcal{E}(g)=\int_{D^2}f(a_g)\left(\lvert \nabla^\perp a_g+\nabla b_g\rvert^2+\lvert \nabla a_g\rvert^2\right)+\frac{1}{4}\lvert \nabla b_g\rvert^2.
\end{align}
Now we claim that
\begin{align}\label{Equation: Integration by parts a,b}
\int_{D^2}f(a_g)\nabla^\perp a_g\nabla b_g=0.
\end{align}
In fact let
$$F: \mathbb{R}\to\mathbb{R},\quad x\mapsto -\frac{1}{2(e^{2x}+1)},$$
then $F'=f$, therefore
$$f(a_g)\nabla^\perp a_g=\nabla^\perp (F\circ a_g).$$
Now for any $\phi\in C^\infty_c(D^2, \mathbb{R})$
\begin{align}\label{Equation: div(f(a)nablaa)=0}
\int_{D^2}f(a_g)\nabla^\perp a_g\nabla \phi=-\int_{D^2}\nabla(F\circ a_g)\nabla^\perp \phi=\int_{D^2}F\circ a_g\, \text{div}(\nabla^\perp \phi)=0.
\end{align}
As $b_g\in W^{1,2}_0(D^2)$ there exists a sequence $(\phi_n)_{n\in \mathbb{N}}$ in $C_c^\infty(D^2)$ such that
$$\phi_n\to b_g\text{ in }W^{1,2}(D^2, \mathbb{R}).$$
Now notice that by Lemma \ref{Lemma: Explicit computation for the energy}
$$\int_{D^2}\lvert f(a_g)\nabla^\perp a_g\rvert^2\leq \int_{D^2}f(a_g)\lvert \nabla a_g\rvert^2<\infty,$$
therefore
$$\int_{D^2}f(a_g)\nabla^\perp a_g\nabla b_g=\lim_{n\to\infty}\int_{D^2}f(a_g)\nabla^\perp a_g\nabla \phi_n=0.$$
This concludes the proof of (\ref{Equation: Integration by parts a,b}).\\
Now by (\ref{Equation: Explicit form of Energy in terms of a,b; 2})
$$\mathcal{E}(g)=\int_{D^2}f(a_g)\left(2\lvert \nabla a_g\rvert^2+\lvert \nabla b_g\rvert^2\right)+\frac{1}{4}\lvert \nabla b_g\rvert^2.$$
Comparing with (\ref{Equation: Explicit form of Energy in terms of a,b; 1}) we obtain
$$\mathcal{E}(g)=\mathcal{E}(\tilde{g})+\int_{D^2}\left(f(a_g)+\frac{1}{4}\right)\lvert \nabla b_g\rvert^2.$$
Then in particular
$$\mathcal{E}(\tilde{g})\leq \mathcal{E}(g).$$
\end{proof}

\subsection{Functions with finitely many topological singularities}

In this subsection we first give a more explicit expression for maps $g$ as in Definition \ref{df-I.1} (and their corresponding $a_g$) when $g$ has finitely many topological singularities in $\overline{D^2}$.
We will then show that for any boundary datum $g_0\in W^{1,1}(\partial D^2, \mathbb{R})$ it is possible to find an extension in $D^2$ with finite renormalized Dirichlet Energy.

\begin{lem}[\textbf{A more explicit form for $a$}]\label{Lemma: Explicit form for a}
Let $g_0\in W^{1,1}(\partial D^2, S^1)$.
Let $Q\in \mathbb{N}$ and for any $i\in \{1,...,Q\}$ let $p_i\in \overline{D^2}$ and $d_i\in \mathbb{Z}$.
Assume that
\begin{align}
\nonumber
i\int_{\partial D^2}g_0^{-1}\partial_{\theta}g_0+2\pi\sum_{\substack{i=1\\ p_i\in D^2}}^Qd_i+\pi\sum_{\substack{i=1\\ p_i\in \partial D^2}}^Q d_i=0.
\end{align}
For any $i\in \{1,...,Q\}$ assume that whenever $p_i\in \partial D^2$, $d_i$ is even.\\
Let $a\in W^{1,1}(D^2, \mathbb{R})$ and assume that for any $\phi\in C^\infty(\overline{D^2})$
\begin{align}\label{Equation: Explicit form Laplacian with deltas and outer normal}
\int_{D^2}\nabla a\nabla\phi =-i\int_{\partial D^2}\phi g_0^{-1}\partial_\theta g_0-2\pi\sum_{\substack{i=1\\ p_i\in D^2}}^Q d_i\phi(p_i)-\pi\sum_{\substack{i=1\\ p_i\in \partial D^2}}^Q d_i\phi(p_i).
\end{align}
Let
\begin{align}\label{Equation: Definition of tilde a}
\Phi(x):=\sum_{i=1}^{Q} d_i\log\lvert x-p_i\rvert\quad\text{and}\quad \tilde{a}(x):=a(x)-\Phi(x)
\end{align}
for any $x\in D^2$. Then
$$a=\tilde{a}+\Phi$$
and for for any $\phi\in C^\infty(\overline{D^2})$
\begin{align}\label{Equation: Condition satisfied by tilde a}
\int_{D^2}\nabla \tilde{a}\nabla \phi=\int_{\partial D^2}\phi\beta,
\end{align}
where
$$
\beta(x)=-i g_0^{-1}\partial_\theta g_0-\sum_{\substack{i=1\\ p_i\in D^2}}^Q d_i\partial_\nu\log\lvert x-p_i\rvert -\frac{1}{2}\sum_{\substack{i=1\\ p_i\in \partial D^2}}^Q d_i .
$$
for any $x\in \partial D^2$. In particular $\tilde{a}$ is harmonic in $D^2$.
\end{lem}
\begin{proof}
Observe that if $p\in D^2$
\begin{align}\label{Equation: Computation effect logarithm 1}
\int_{D^2}\nabla \phi\nabla \log\lvert x-p\rvert&=\lim_{\varepsilon\to 0}\int_{\partial(D^2\smallsetminus B_\varepsilon(p))}\phi\,\partial_\nu \log\lvert x-p\rvert\\
\nonumber
&=\int_{\partial D^2}\phi\,\partial_\nu \log\lvert x-p\rvert-\lim_{\varepsilon\to 0}\int_{\partial B_\varepsilon(p)}\phi \frac{1}{\varepsilon}\\
\nonumber
&=\int_{\partial D^2}\phi\,\partial_\nu \log\lvert x-p\rvert-2\pi\phi(p).
\end{align}
On the other hand, if $p\in \partial D^2$
\begin{align}\label{Equation: Computation effect logarithm 2}
\int_{D^2}\nabla \phi\nabla \log\lvert x-p\rvert&=\lim_{\varepsilon\to 0}\int_{\partial (D^2\smallsetminus B_\varepsilon(p))}\phi\,\partial_\nu \log\lvert x-p\rvert\\
\nonumber
&=\lim_{\varepsilon\to 0}\int_{\partial D^2\smallsetminus B_\varepsilon(p)}\phi\,\frac{(x-p)\cdot x}{\lvert x-p\rvert^2}-\lim_{\varepsilon\to 0}\int_{\partial B_\varepsilon(p)\cap D^2}\phi \frac{1}{\varepsilon}\\
\nonumber
&=\frac{1}{2}\int_{\partial D^2}\phi-\pi\phi(p)
\end{align}
where we used the fact that for any $x, p\in \partial D^2$
\begin{align}
\nonumber
\frac{(x-p)\cdot x}{\lvert x-p\rvert^2}=\frac{1-x\cdot p}{\lvert x-p\rvert^2}=\frac{p\cdot (p-x)}{\lvert x-p\rvert^2}
\end{align}
and therefore
\begin{align}
\nonumber
\frac{(x-p)\cdot x}{\lvert x-p\rvert^2}=\frac{1}{2}\frac{(x-p)\cdot(x-p)}{\lvert x-p\rvert^2}=\frac{1}{2}.
\end{align}
Thus the function $\tilde{a}$ defined in (\ref{Equation: Definition of tilde a}) satisfies (\ref{Equation: Condition satisfied by tilde a}).
\end{proof}
\begin{rem}
Whenever $a\in W^{1,p}(D^2)$ satisfies (\ref{Equation: Explicit form Laplacian with deltas and outer normal}) we will say that $a$ is a solution of
\begin{align*}
\begin{cases}
\Delta a=2\pi\sum_{i=1}^Q d_i\delta_{p_i}&\text{ in }\overline{D^2}\\[5mm]
\partial_\nu a=-ig_0^{-1}\partial_\theta g_0&\text{ on }\partial D^2.
\end{cases}
\end{align*}
\end{rem}
\begin{cor}[\textbf{A more explicit form for $g$}]\label{Lemma: Explicit form for g}
Let $g$ be a $S^1$-valued map in $W^{1,p}$ (for some $p>1$) with finitely many topological singularities in $\overline{D^2}$. Assume that $a_g$ in the Hodge decomposition (\ref{Equation: Definition of ag}) satisfies the assumptions of Lemma \ref{Lemma: Explicit form for a}. Then $g$ has the following form:
\begin{align}
\nonumber
g(z)=\prod_{i=1}^Q\left(\frac{z-p_i}{\lvert z-p_i\rvert}\right)^{-d_i}e^{i\varphi},
\end{align}
where
$$\varphi=\mathcal{H}(\tilde{a}_g)+b_g$$
up to an additive constant. Here $\tilde{a}$ is the function introduced in (\ref{Equation: Definition of tilde a}) and $\mathcal{H}(\tilde{a})$ denotes the harmonic conjugate of $\tilde{a}$ in $D^2$, i.e.
\begin{align*}
\nabla \mathcal{H}(\tilde{a}_g)=\nabla^\perp \tilde{a}_g\text{ and }\mathcal{H}(\tilde{a}_g)(0)=0.
\end{align*}
\end{cor}
\begin{proof}
The result follows from Lemma \ref{Lemma: Explicit form for a} and the fact that for any $p\in \overline{D^2}$
\begin{align}\label{Equation: Relationship between d and star of logs}
\frac{\nabla\left(\frac{z-p}{\lvert z-p\rvert}\right)}{\frac{z-p}{\lvert z-p\rvert}}
=i\frac{\nabla^\perp\lvert z-p\rvert}{\lvert z-p\rvert}.
\end{align}
\end{proof}

The following Lemma shows that for any given $g_0\in W^{1,1}(\partial D^2, S^1)$ the class of functions $g\in W^{1,(2,\infty)}(D^2, S^1)$ with $g\vert_{\partial D^2}=g_0$, with finitely many topological singularities and with
$$\mathcal{E}(g)<\infty$$
is not empty.

\begin{lem}\label{Lemma: Existence of a competitor}
Let $d\in \mathbb{Z}$. Let $g_0\in W^{1,1}(\partial D^2, S^1)$ with $\deg(g_0)=d$. Then there exist $g\in W^{1,(2,\infty)}(D^2, S^1)$ as in Definition \ref{df-I.1} with finitely many topological singularities such that $g\vert_{\partial D^2}=g_0$ and
$$\mathcal{E}(u_g)<\infty.$$
More precisely, there exists a map
\[
Ext_d: W^{1,1}_{\deg=d}(\partial D^2, S^1)\to \dot{W}^{1,2}(D^2, S^2) 
\]
sending a boundary datum $g_0$ to a function $u_g$ as in (\ref{Equation: Definition of the map ug}) corresponding to a function $g\in W^{1, (2,\infty)}(D^2, S^1)$ with finitely many topological singularities, with $g\vert_{\partial D^2}=g_0$, and so that
\begin{align*}
\lVert \nabla g\rVert_{L^{2,\infty}}\leq C\left(\lVert g_0\rVert_{W^{1,1}}+\lvert d\rvert\right)
\end{align*}
for some constant $C$ and
\begin{align*}
\frac{1}{2}\int_{D^2}\lvert \nabla u_g\rvert^2\leq \frac{\pi^2}{2}\lVert \partial_\theta g_0\rVert_{L^1(\partial D^2)}+4\pi\lvert d\rvert.
\end{align*}
\end{lem}
\begin{proof}
Let $g_0\in W^{1,1}(\partial D^2, S^1)$ with $\deg(g_0)=0$. Then there exists a lift $\phi_0\in C^\infty(S^1, \mathbb{R})$ such that
$$g_0=e^{i\phi_0}\text{ on }\partial D^2.$$
Let $\phi$ be the solution of the Cauchy problem
\[
\begin{cases}
\Delta \phi=0&\text{ in }D^2\\[5mm]
\phi=\phi_0&\text{ on }\partial D^2.
\end{cases}
\]
and let $g= e^{i\phi}$ in $D^2$ . Then by Lemma \ref{Lemma: continuous extension from W11 to W 12infty} $g\in W^{1,(2,\infty)}(D^2, S^1)$ with
\begin{align*}
\lVert \nabla g\rVert_{L^{2,\infty}}\leq C\lVert g_0\rVert_{W^{1,1}}
\end{align*}
ang $g$ has no topological singularities in $\overline{D^2}$.
Let $a: \overline{D^2}\to \mathbb{R}$ so that
$$\int_{D^2}a=0\,\text{ and }\,\nabla \phi=\nabla^\perp a$$
and let
$$u_g=\pi^{-1}(e^{a} g)=\pi^{-1}(e^{a+i\phi})$$ as in (\ref{Equation: Definition of the map ug}).
Let's consider first the case where $g_0\in C^\infty(\partial D^2, S^1)$. In this case all the functions considered so far are smooth and the map $u_g$ (and its continuous extension to $\overline{D^2}$, which we will also denote by $u_g$) does not take the south poles of $S^2$ as a value.\\
Let 
\begin{align}\label{Equation: Definition of the auxiliary function A}
A: \overline{D^2}\to S^2, \quad (r, \varphi)\mapsto\left(g_0(e^{i\varphi}), r \theta(u_g\vert_{\partial D^2}(e^{i\varphi}))\right),
\end{align}
here the first coordinate of $S^2$ is the azimuth angle as an element of the equator, while the second is the polar angle (measured in radians with respect to the north pole), and $\theta(u_g\vert_{\partial D^2}(e^{i\varphi}))$ denotes the polar angle of the point $u_g\vert_{\partial D^2}(e^{i\varphi})$. So $A$ is a parametrization of one of the two connected components of $S^2$ delimited by $u_g\big\vert_{\partial D^2}$, the one containing the north pole.\\
Note that $A$ is a continuous function and it coincides with $u_g$ on $\partial D^2$.
Let
\[S^2_+:=S^2\cap\left\{(x_1, x_2, x_3)\in \mathbb{R}^3\vert x_3\geq 0\right\}, \quad S^2_-:= S^2\cap\left\{(x_1, x_2, x_3)\in \mathbb{R}^3\vert x_3\leq 0\right\}
\]
and let
\[
\mu_+: S^2_+\to \overline{D^2}, \quad \mu_-: S^2_-\to \overline{D^2}
\]
be two smooth diffeomorphisms, the first one orientation-preserving and the second one orientation-reversing, both equal to the projection to the first two components if restricted to $\partial D^2\times\{0\}$.\\
Then the maps $u_g\circ \mu_+$ and $A\circ \mu_-$ can be glued along $\partial D^2\times\{0\}$ to obtain a Lipschitz continuous map
$$F: S^2\to S^2.$$
Observe that $\deg(F)=0$, as a neighbourhood the south pole of $S^2$ does not belong to the image of $F$.\\
Therefore
\begin{align*}
\int_{S^2}F^*dvol_{S^2}=0
\end{align*}
and so 
\begin{align*}
\int_{S^2_+}(u_g\circ\mu_+)^\ast dvol_{S^2}=\int_{S^2_-}(A\circ\mu_-)^\ast dvol_{S^2}.
\end{align*}
Now since $u_g$ is holomorphic and $\mu_+$ is orientation preserving,
\begin{align*}
\int_{S^2}\#(u_g\circ \mu_+)^{-1}(y)dvol_{S^2}=\int_{S^2_+}(u_g\circ \mu_+)^\ast dvol_{S^2}=-\int_{S^2_-}(A\circ\mu_-)^\ast dvol_{S^2}.
\end{align*}
Here and in the following the symbol $\#$ denotes the cardinality of a set.\\
We also have
\[
-\int_{S^2_-}(A\circ \mu_-)^\ast dvol_{S^2}=-\deg(\mu_-)\int_{D^2}A^\ast dvol_{S^2}\leq \int_{D^2}\lvert J A\rvert.
\]
To estimate the last term, let's introduce the following function:
\begin{align}\label{Equation: Definition of the second auxiliary function A}
\overline{A}: D^2\to S^2, \quad (r, \varphi)\mapsto \left(g_0(e^{i\varphi}), r\pi\right).
\end{align}
(here we are using the same coordinates as in (\ref{Equation: Definition of the auxiliary function A})).
Then for any $y\in S^2$
\[
\# A^{-1}(y)\leq \# \overline{A}^{-1}(y).
\]
Therefore, by the area formula,
\begin{align}\label{Equation: Estimate integral JA 1}
\int_{D^2}\lvert J A\rvert dx^2=\int_{S^2}\# A^{-1}(y)dvol_{S^2}\leq \int_{S^2}\#\overline{A}^{-1}(y)dvol_{S^2}=\int_{D^2}\lvert J \overline{A}\rvert dx^2.
\end{align}
One computes that
\begin{align*}
\lvert J \overline{A}(r, \theta)\rvert \leq \pi^2
\left\lvert\partial_\theta g_0(e^{i\theta})\right\rvert,
\end{align*}
therefore
\begin{align*}
\int_{D^2}\lvert J\overline{A}\rvert dx^2\leq \pi^2\int_0^1 r\int_{\partial D^2}\lvert \partial_\theta g_0\rvert d\theta dr=\frac{\pi^2}{2}\lVert \partial_\theta g_0\rVert_{L^1(\partial D^2)}
\end{align*}
Since $u_g$ is holomorphic, we conclude that
\begin{align}\label{Equation: Esimate on the energy in terms of W11}
\frac{1}{2}\int_{D^2}\lvert \nabla u_g\rvert^2dx^2=\int_{S^2}\#u_g^{-1}(y)dvol_{S^2}=\int_{S^2}\#(u_g\circ\mu_+)^{-1}(y)dvol_{S^2}\leq \frac{\pi^2}{2}\lVert \partial_\theta g_0\rVert_{L^1(\partial D^2)}.
\end{align}
Therefore the procedure described above induces a bounded continuous map $$Ext_0: W^{1,1}_{\deg=0}(\partial D^2, S^1)\to \dot{W}^{1,2}(D^2, S^2),\quad g_0\mapsto u_g$$
In fact, given a generic map $g_0\in W^{1,1}_{\operatorname{deg}=0}(\partial D^2, S^1)$ let $(g_0^n)_{n\in \mathbb{N}}$ be a sequence of degree zero maps in $C^\infty(\partial D^2, S^1)$ such that
$$g_0^n\to g_0\text{ in }W^{1,1}(\partial D^2).$$
Then by Lemma \ref{Lemma: continuous extension from W11 to W 12infty}
$$g_n\to g\text{ and }a_n\to a\text{ in }W^{1,(2,\infty)}(D^2).$$
In particular, up to a subsequence,
$$u_{g_n}\to u_g\text{ a.e..}$$
Upon considering a further subsequence, the weak lower semi-continuity of the norm implies that estimate (\ref{Equation: Esimate on the energy in terms of W11}) passes to the limit and thus holds for $u_g$.\\
Next let's consider the case where $g_0\in W^{1,1}(\partial D^2, S^1)$ and $\deg(g_0)=d$ for some $d\in \mathbb{Z}$.
Let
$$\tilde{g}_0=\left(\frac{z}{\lvert z\rvert}\right)^{-d}g_0\text{ on }\partial D^2.$$
Let $\tilde{\phi}$ be its harmonic extension in $D^2$ and set
$$g:=\left(\frac{z}{\lvert z\rvert}\right)^de^{i\tilde{\phi}}.$$
Then $g\vert_{\partial D^2}=g_0$ and by Lemma \ref{Lemma: continuous extension from W11 to W 12infty}
\begin{align}\label{Equation: Estimate on the L2weak norm of grad g}
\lVert \nabla g\rVert_{L^{2,\infty}}\leq C\left(\lvert d\rvert+\lVert\partial_\theta g_0\rVert_{L^1(\partial D^2)}\right).
\end{align}
Notice that the corresponding function $a_g$ in the decomposition (\ref{Equation: Definition of ag}) is given by
$$a_g=d\log\lvert z\rvert-\mathcal{H}(\tilde{\phi}),$$
where $\mathcal{H}(\tilde{\phi})$ is the harmonic conjugate of $\tilde{\phi}$, therefore
\begin{align}\label{Equation: Bound for the 2,infty norm of a, depending on d}
\lVert a_g\rVert_{L^{2,\infty}}\leq C\left(\lvert d\rvert+\lVert\partial_\theta \tilde{g_0}\rVert_{L^1(\partial D^2)}\right).
\end{align}
As above let
$$u_g=\pi^{-1}(e^{a_g}g)\text{ in }D^2.$$
Let's assume now that $g_0\in C^\infty(\partial D^2, S^1)$
and let $A$ be the map introduced in (\ref{Equation: Definition of the auxiliary function A}).\\
Again the maps $u_g\circ \mu_+$ and $A\circ \mu_-$ can be glued along $\partial D^2\times \{0\}$ to obtain a Lipschitz continuous map
$$F: S^2\to S^2.$$
Now $\deg(F)=d$, as one can see considering the preimages of point around the south pole (if $d$ is negative) or the north pole (if $d$ is positive).\\
Therefore
\begin{align*}
\int_{S^2}F^*dvol_{S^2}=d4\pi
\end{align*}
Thus, arguing as above, we obtain
$$\int_{S^2}\#(u_g\circ \mu_+)^{-1}(y) dvol_{S^2}\leq \int_{D^2}\lvert JA\rvert dx^2 +4\pi\lvert d\rvert.$$
As estimate (\ref{Equation: Estimate integral JA 1}) remains true for the function $\overline{A}$ introduced in (\ref{Equation: Definition of the second auxiliary function A}), we conclude that
\begin{align}\label{Equation: Final estimate for extension operator}
\frac{1}{2}\int_{D^2}\lvert \nabla u_g\rvert^2dx^2=\int_{S^2}\#u_g^{-1}(y)dvol_{S^2}=\int_{S^2}\#(u_g\circ \mu_+)^{-1}(y) dvol_{S^2}\leq \frac{\pi^2}{2}\lVert \partial_\theta g_0\rVert_{L^1(\partial D^2)}+4\pi\lvert d\rvert.
\end{align}
Just as above one can verify that the prescription
$$g_0\mapsto u_g$$
induces a bounded map
$$Ext_d: W^{1,1}_{\deg=d}(\partial D^2, S^1)\to \dot{W}^{1,2}(D^2, S^2)$$
such that for any $g_0\in W^{1,1}(\partial D^2, S^1)$ the corresponding $u_g$ is the "$S^2$ lift" of a map $g\in W^{1, (2,\infty)}(\partial D^2, S^1)$ with finitely many topological singularities in $\overline{D^2}$ such that
estimates (\ref{Equation: Estimate on the L2weak norm of grad g}) and (\ref{Equation: Final estimate for extension operator}) hold true.\\
\end{proof}

\subsection{Stability of the renormalized Dirichlet Energy}
Next we show that when the boundary datum $g_0$ lies in $H^\frac{1}{2}(\partial D^2)$, the renormalized energy is stable under displacements of the topological singularities, even if a topological singularity is pushed to the boundary.
\begin{lem}\label{Lemma: Stability wrt points}
Let $g_0\in H^\frac{1}{2}(\partial D^2, S^1)$, let $Q\in \mathbb{N}$. Let $p_i\in \overline{D^2}$ and $d_i\in \mathbb{Z}\smallsetminus\{0\}$ for any $i\in \{1,..., Q\}.$ For any $i\in \{1,...,Q\}$ if $p_i\in \partial D^2$ assume that $d_i$ is even.\\
Let $(p_1^k)_{k\in \mathbb{N}}$ be a sequence of points in $\overline{D^2}$ such that $p_1^k\to p_1$.
For any $k\in \mathbb{N}$ let $a_k$ denote the zero-average solution of
\begin{align*}
\begin{cases}
\Delta a_k=2\pi\left(d_1\delta_{p_1^k} +\sum_{i=2}^Q d_i \delta_{p_i}\right)&\text{ in }\overline{D^2}\\[5mm]
\partial_{\nu}a_k=-ig_0^{-1}\partial_\theta g_0&\text{ on }\partial D^2
\end{cases}
\end{align*}
and let $a$ denote the zero-average solution of
\begin{align*}
\begin{cases}
\Delta a=2\pi\left(d_1\delta_{p_1} +\sum_{i=2}^Q d_i \delta_{p_i}\right)&\text{ in }\overline{D^2}\\[5mm]
\partial_{\nu}a=-ig_0^{-1}\partial_\theta g_0&\text{ on }\partial D^2.
\end{cases}
\end{align*}
Then
\begin{align*}
\lim_{k\to \infty}\int_{D^2}f(a_k)\lvert \nabla a_k\rvert^2=\int_{D^2}f(a)\lvert \nabla a\rvert^2.
\end{align*}
\end{lem}

\begin{proof}
\
\\
\textbf{Claim 1:}
\begin{align*}
-ig_0^{-1}\partial_\theta g_0\in H^{-\frac{1}{2}}(\partial D^2).
\end{align*}
\begin{proof}[Proof of Claim 1]
Let
$$d:=\deg(g_0)$$
and let
$$\tilde{g_0}(e^{i\theta}):=g_0(e^{i\theta})e^{-id\theta}\quad \forall \theta\in [0,2\pi).$$
Then $\tilde{g_0}\in H^\frac{1}{2}$ and
$$\deg(\tilde{g_0})=0.$$
Therefore by Theorem 1 in \cite{BBM} there exists a function $\varphi_0\in H^\frac{1}{2}(\partial D^2)$ such that 
$$\tilde{g_0}=e^{i\varphi_0}.$$
Now we claim that
\begin{align}\label{Equation: Claim: derivative of lift}
\partial_\theta \varphi_0=-ie^{-i\varphi_0}\partial_\theta e^{i\varphi_0}.
\end{align}
To see this let $(\varphi_n)_{n\in \mathbb{N}}$ be a sequence in $C^\infty(\partial D^2)$ such that
\begin{align*}
\varphi_n\to\varphi_0\text{ in }H^\frac{1}{2}(\partial D^2).
\end{align*}
Then for any $n\in \mathbb{N}$
\begin{align*}
\partial_\theta \varphi_n=-ie^{-i\varphi_n}\partial_\theta e^{i\varphi_n}.
\end{align*}
By Lemma \ref{Lemma: H12 convergence for composition with Lipschitz} there holds
\begin{align*}
e^{i\varphi_n}\to e^{i\varphi_0}\text{ in }H^\frac{1}{2}(\partial D^2).
\end{align*}
Therefore
\begin{align*}
\partial_\theta e^{i\varphi_n}\to \partial_\theta e^{i\varphi_0} \text{ in }H^{-\frac{1}{2}}(\partial D^2)
\end{align*}
and so
\begin{align*}
e^{i\varphi_n}\partial_\theta e^{i\varphi_n}\to e^{-i\varphi_0}\partial_\theta e^{i\varphi_0}\text{ in }\mathcal{D}'(\partial D^2).
\end{align*}
On the other hand
\begin{align*}
\partial_\theta \varphi_n\to\partial_\theta\varphi_0\text{ in }H^{-\frac{1}{2}}(\partial D^2),
\end{align*}
therefore (\ref{Equation: Claim: derivative of lift}) follows.\\
Now we compute
\begin{align*}
e^{-i\varphi_0}\partial_\theta e^{i\varphi_0}=\tilde{ g_0}^{-1}\partial_\theta \tilde{ g_0}=g_0^{-1}e^{-id\theta}\partial_\theta(g_0 e^{id\theta})=g_0^{-1}\partial_\theta g_0+id.
\end{align*}
As $$e^{i\varphi_0}\partial_\theta e^{i\varphi_0}\in H^{-\frac{1}{2}}(\partial D^2)$$ by (\ref{Equation: Claim: derivative of lift}) and clearly $id\in H^{-\frac{1}{2}}(\partial D^2)$ we conclude that $$-g_0^{-1}\partial_\theta g_0\in H^{-\frac{1}{2}}(\partial D^2).$$
\end{proof}
Let $\delta>0$ (to be determined later) and let $h\in C^{\infty}(\partial D^2, S^1)$ such that
\begin{align*}
\int_{\partial D^2}h=-i\int_{\partial D^2}g_0^{-1}\partial_\theta g_0
\end{align*}
and
\begin{align*}
\lVert -ig_0^{-1}\partial_\theta g_0-h\rVert_{H^{-\frac{1}{2}}(\partial D^2)}<\delta.
\end{align*}
Let $a_1$, $a_2$, $a_3$ be zero-mean solutions of
\begin{align}\label{Equation: Cauchy problem solved by a1}
\begin{cases}
\Delta a_1=2\pi\sum_{i=1}^Q d_i\delta_{p_i}-\sum_{i=1}^Q d_i&\text{ in }\overline{D^2}\\[5mm]
\partial_{\nu}a_1=0&\text{ on }\partial D^2,
\end{cases}
\end{align}
\begin{align*}
\begin{cases}
\Delta a_2=\sum_{i=1}^Q d_i&\text{ in }D^2\\[5mm]
\partial_{\nu}a_2=h&\text{ on }\partial D^2,
\end{cases}
\end{align*}
\begin{align*}
\begin{cases}
\Delta a_3=0&\text{ in }D^2\\[5mm]
\partial_{\nu}a_3=-i g_0^{-1}\partial_\theta g_0-h&\text{ on }\partial D^2.
\end{cases}
\end{align*}
Then $a=a_1+a_2+a_3$.\\
Notice that $a_2$ is smooth, $a_3$ lies in $H^1(D^2)$ with
\begin{align*}
\lVert a_3\rVert_{H^1}\leq\delta
\end{align*}
and only $a_1$ depends on the positions and the degrees of the topological singularities.\\
For any $k\in \mathbb{N}$ let $a_1^k$, $a_2^k$, $a_3^k$ be defined analogously and observe that for any $k\in \mathbb{N}$ $a_2^k=a_2$ and $a_3^k=a_3$.\\

\textbf{Claim 2}:
\begin{align}\label{Equation: explicit expression for a_1}
a_1(x)=\sum_{i=1}^Q d_i\left(\log\lvert x-p_i\rvert +\log\lvert x-\overline{p_i}^{-1}\rvert-\frac{1}{2}\lvert x\rvert^2\right)
\end{align}
up to an additive constant, with the convention that if $p=0$
\begin{align*}
\log\lvert x-\overline{p}^{-1}\rvert\equiv0.
\end{align*}
\begin{proof}[Proof of Claim 2]
First we observe that by linearity it is enough to check the Claim for $Q=1$ and $d_1=1$. Let $p\in \overline{D^2}$ denote the only singularity of $a_1$.
Let $\tilde{a_1}$ denote the function defined by the right hand side of (\ref{Equation: explicit expression for a_1}).\\
If $p=0$ the Claim is clear. If $p\in \partial D^2$ then $\overline{p}^{-1}=p$ and
\begin{align*}
\tilde{a_1}(x)=2\log\lvert x-p\rvert-\frac{1}{2}.
\end{align*}
By Computation (\ref{Equation: Computation effect logarithm 2}) for any $\phi\in C^\infty(\overline{D^2})$ there holds
\begin{align*}
\int_{D^2}\nabla\phi\nabla \tilde{a_1}=\int_{D^2}\phi-2\pi\phi(p),
\end{align*}
then $\tilde{a_1}$ is a solution of (\ref{Equation: Cauchy problem solved by a1}) and thus it differs from $a_1$ at most by an additive constant.\\
Let's consider the case where $p\in D^2$ and $p\neq 0$.
It is clear that $\tilde{a_1}$ satisfies
\begin{align*}
\Delta \tilde{a_1}=2\pi \delta_p-1\text{ in }D^2.
\end{align*}
We still need to check that $\tilde{a_1}$ also satifies the Neumann boundary condition satisfied by $a_1$. Let
\begin{align*}
\tau: \mathbb{C}\smallsetminus\{1\}\to \mathbb{C},\quad z\mapsto\frac{1+z}{1-z}.
\end{align*}
Observe that $\tau$ restricts to a biholomorphic map from a neighbourhood of $\overline{D^2}\smallsetminus\{1\}$ to a neighbourhood of $\overline{\mathbb{H}}$, whose inverse is given by
\begin{align*}
\tau^{-1}: \mathbb{C}\smallsetminus\{-1\}\to \mathbb{C},\quad w\mapsto\frac{w-1}{w+1}.
\end{align*}
Notice that for any $p\in \mathbb{C}\smallsetminus\{-1\}$
\begin{align*}
\overline{p}^{-1}=\tau^{-1}\left(-\overline{\tau(p)}\right).
\end{align*}
Now set
\begin{align*}
F_p: \mathbb{C}\to\mathbb{C}, z\mapsto \log\lvert z-\tau(p)\rvert+\log\lvert z+\overline{\tau(p)}\rvert.
\end{align*}
For any $z\in \mathbb{C}\smallsetminus\{1\}$
\begin{align*}
F_p\circ \tau (z)=&\log\lvert 1+z-\tau(p)(1-z)\rvert+\log\lvert 1+z+\overline{\tau(p)}(1-z)\rvert-2\log\lvert 1-z\rvert\\
=&\log\left\lvert z+\frac{1-\tau(p)}{1+\tau(p)}\right\rvert+\log\lvert 1+\tau(p)\rvert+\log\left\lvert z+\frac{1+\overline{\tau(p)}}{1-\overline{\tau(p)}}\right\rvert+\log\lvert1-\overline{\tau(p)}\rvert-2\log\lvert 1-z\rvert\\
=&\tilde{a_1}(z)-\frac{1}{2}\lvert z\rvert^2+\log\lvert 1+\tau(p)\rvert+\log\lvert 1-\overline{\tau(p)}\rvert-2\log\lvert 1-z\rvert.
\end{align*}
Recall that
\begin{align*}
\partial_\nu \log\lvert 1-z\rvert=\frac{1}{2}\text{ on }\partial D^2\smallsetminus\{1\},
\end{align*}
as shown in (\ref{Equation: Computation effect logarithm 2}). Moreover
\begin{align*}
\partial_\nu (F_p\circ \tau) (z)=DF_p(\tau(z))\partial_\nu\tau (z)=0\text{ on }\partial D^2\smallsetminus\{1\}
\end{align*}
since $\partial_\nu\tau (z)$ is orthogonal to the imaginary axis, and thus $DF_p$ vanishes in that direction.
Therefore
\begin{align*}
\partial_\nu \left(\tilde{a_1}(z)-\frac{1}{2}\lvert z\rvert^2\right)=\partial_\nu (F_p\circ \tau)(z)+2\partial_\nu\log\lvert 1-x\rvert=1
\end{align*}
on $\partial D^2\smallsetminus\{1\}$.\\
Now
\begin{align*}
\partial_\nu \frac{1}{2}\lvert x\rvert^2=1\text{ on }\partial D^2,
\end{align*}
therefore
\begin{align*}
\partial_\nu \tilde{a_1}=0\text{ on }\partial D^2.
\end{align*}
\end{proof}
Claim 2 implies that
\begin{align*}
a_1(x)=&\sum_{i=1}^Q d_i\left(\log\lvert x-p_i\rvert-\fint_{ D^2}\log\lvert y-p_i\rvert dy\right)\\&+\sum_{i=1}^Q d_i\left(\log\lvert x-\overline{p_i}^{-1}\rvert-\fint_{ D^2}\log\lvert y-\overline{p_i}^{-1}\rvert dy\right)+\left(\frac{1}{2}\lvert x\rvert^2-\frac{\pi}{4}\right)\sum_{i=1}^Q d_i.
\end{align*}
Observe that the analogous result holds for $a_1^k$ for any $k\in \mathbb{N}$.\\

\textbf{Claim 3}: There exists a constant $C$ such that for any $k\in \mathbb{N}$
\begin{align}\label{Equation: Integrable bound}
f(a_k)\lvert \nabla a_k\rvert^2\leq C\left(e^{2a_3}+\lvert\nabla(a_2+a_3)\rvert^2\right)
\end{align}
\begin{proof}[Proof of Claim 3]
Let $k\in \mathbb{N}$. To simplify the notation, in the proof of Claim 3 we will set $p_1=p_1^k$.
Notice first that
\begin{align*}
f(a_k)\lvert\nabla a_k\rvert^2\leq 2\left( f(a_k)\lvert \nabla a_k^1\rvert^2+\lvert \nabla (a_k^2+a_k^3)\rvert^2\right),
\end{align*}
therefore it is enough to show that there exists a constant $C$, independent from $k$, such that
\begin{align*}
f(a_k)\lvert \nabla a_k^1\rvert^2\leq Ce^{2 a_3}.
\end{align*}
The key step will consist in proving the following estimate: for any  $i\in \{1,..., Q\}$, for any $x\in D^2$
\begin{align}\label{Equation: first estimate for products}
\prod_{j=1}^Q\left(\lvert x-\overline{p_j}^{-1}\rvert^{2 d_j}\operatorname{exp}\left(-\fint_{ \partial D^2}\log\lvert y-\overline{p_j}^{-1}\rvert^{2 d_j}dy\right)\right)\leq C\min\left(1, \lvert x-\overline{p_i}^{-1}\rvert^{2d_i}\right)
\end{align}
for some constant $C$ depending only on $Q$ and the degrees $d_1,..., d_Q$.\\
If $\overline{p_i}^{-1}\in B_2(0)$,
\begin{align*}
\left\lvert\fint_{\partial D^2}\log\lvert y-\overline{p_i}^{-1}\rvert dy\right\rvert \leq \frac{1}{2\pi}\int_{B_4(0)}\left\lvert\log\lvert y\rvert\right\rvert dy
\end{align*}
and therefore
\begin{align}\label{Equation: first estimate for products, estimate 1}
\lvert x-\overline{p_i}^{-1}\rvert^{2d_i}\operatorname{exp}\left(-\fint_{\partial D^2}\log\lvert y-\overline{p_i}^{-1}\rvert^{2d_i}dy\right) \frac{1}{\lvert x-\overline{p_i}^{-1}\rvert^{2d_i}}\leq C
\end{align}
for some constant $C$ independent from $p_i$, for any $x\in D^2$.\\
On the other hand if $\overline{p_i}^{-1}\notin B_2(0)$
\begin{align*}
\lvert x-\overline{p_i}^{-1}\rvert^{2 d_i}\geq 1
\end{align*}
and
\begin{align*}
\left\lvert \log \lvert x-\overline{p_i}^{-1}\rvert^{2 d_i}-\fint_{\partial D^2}\log \lvert y-\overline{p_i}^{-1}\rvert^{2 d_i} dy\right\rvert\leq 4\lvert d_i\rvert\sup_{y\in D^2}\frac{1}{\lvert y-\overline{p_i}^{-1}\rvert}\leq 4 \lvert d_i\rvert
\end{align*}
for any $x\in D^2$.
Thus in this case
\begin{align}\label{Equation: first estimate for products, estimate 2}
\operatorname{exp}\left(\log \lvert x-\overline{p_i}^{-1}\rvert^{2 d_i}-\fint_{\partial D^2}\log \lvert y-\overline{p_i}^{-1}\rvert^{2 d_i} dy\right)\leq e^{4\lvert d_i\rvert}
\end{align}
for any $x\in D^2$.\\
Combining (\ref{Equation: first estimate for products, estimate 1}) and (\ref{Equation: first estimate for products, estimate 2}) we obtain (\ref{Equation: first estimate for products}).\\
We also have for any $i\in \{1,..., Q\}$, for any $x\in D^2$
\begin{align*}
\prod_{j=1}^Q\left(\lvert x-p_j\rvert^{2d_j}\operatorname{exp}\left(-\fint_{ D^2}\log\lvert y-p_j\rvert^{2d_j} dy\right)\right)\leq C\min\left(1, \lvert x-p_i\rvert^{2d_i}\right)
\end{align*}
for some constant $C$ depending only on $Q$ and the degrees $d_1,..., d_Q$.\\
Now for any $x\in D^2$
\begin{align*}
\nabla a_k^1(x)=\sum_{i=1}^Q d_i\left(\frac{x-p_i}{\lvert x-p_i\rvert^2}+\frac{x-\overline{p_i}^{-1}}{\lvert x-\overline{p_i}^{-1}\rvert^2}-x\right),
\end{align*}
therefore we conclude that for any $x\in D^2$
\begin{align*}
e^{2a_1}\lvert\nabla a_k^1\rvert^2=\prod_{i=1}^Qe^{d_i\lvert x\rvert^2}\lvert x-p_i\rvert^{2d_i}\lvert x-\overline{p_i}^{-1}\rvert^{2d_i}\lvert \nabla a_k^1\rvert^2\leq C
\end{align*}
for some constant $C$ depending only on $Q$ and the degrees $d_1,..., d_Q$.\\
Then
\begin{align*}
f(a_k)\lvert\nabla a_k\rvert^2\leq e^{2a_k}\lvert \nabla a_k\rvert^2 \leq C\lVert e^{2a_2}\rVert_{L^\infty}e^{2a_3}.
\end{align*}
\end{proof}
We notice that the right hand side of (\ref{Equation: Integrable bound}) is integrable if $\delta$ is chosen to be sufficiently small. In fact $a_2$ and $a_3$ lie in $H^1(D^2)$, moreover since $a_3\in H^1(D^2)\subset BMO(D^2)$, by the John-Nirenberg Theorem (see Corollary 3.1.7 in \cite{Gra2})
\begin{align*}
\int_{D^2}e^{2a_3(x)}dx<\infty
\end{align*}
if $\delta$ is chosen sufficiently small (and thus the $BMO$-norm of $a_3$ is sufficiently small).\\

\textbf{Claim 4}:
\begin{align*}
f(a_k)\lvert \nabla a_k\rvert^2\to f(a)\lvert \nabla a\rvert^2\text{ a.e.}.
\end{align*}
\begin{proof}[Proof of Claim 4]
It is enough to check that
\begin{align}\label{Equation: Limit of log a e}
&\log\lvert x-p_1^k\rvert+\log\left\lvert x-\overline{p_1^k}^{-1}\right\rvert-\fint_{ D^2}\left(\log\lvert y-p_1^k\rvert +\log\left\lvert y-\overline{p_1^k}^{-1}\right\rvert \right)dy\\
\nonumber
&\to \log\left\lvert x-p_1^k\right\rvert+\log\left\lvert x-\overline{p_1}^{-1}\right\rvert-\fint_{ D^2}\left(\log\left\lvert y-p_1^k\right\rvert +\log\left\lvert y-\overline{p_1}^{-1}\right\rvert \right)dy\text{ a.e.}
\end{align}
and
\begin{align}\label{Equation: Limit of gradient of log a e}
\frac{x-p_1^k}{\left\lvert x-p_1^k\right\rvert^2}+ \frac{x-\overline{p_1^k}^{-1}}{\left\lvert x-\overline{p_1^k}^{-1} \right\rvert^2}\to \frac{x-p_1}{\left\lvert x-p_1\right\rvert^2}+\frac{x-\overline{p_1}^{-1}}{\left\lvert x-\overline{p_1}^{-1} \right\rvert^2}\text{ a.e.}
\end{align}
with the convention that if $p=0$ then
\begin{align*}
\log\left\lvert x-\overline{p}^{-1}\right\rvert\equiv0\text{ and }\frac{x-\overline{p}^{-1}}{\left\lvert x-\overline{p}^{-1} \right\rvert^2}\equiv0.
\end{align*}
When $p_1\neq 0$ both (\ref{Equation: Limit of log a e}) and (\ref{Equation: Limit of gradient of log a e}) are clear.\\
When $p_1=0$ then we can assume without loss of generality that $p_1^k\neq 0$ for any $k\in \mathbb{N}$. In order to show (\ref{Equation: Limit of log a e}) in this case we have to check that
\begin{align*}
\log\left\lvert x-\overline{p_1^k}^{-1}\right\rvert -\fint_{ D^2}\log\left\lvert y-\overline{p_1^k}^{-1}\right\rvert dy\to 0\text{ a.e..}
\end{align*}
In fact
\begin{align*}
\left\lvert\fint_{D^2}\left(\log\left\lvert x-\overline{p_1^k}^{-1}\right\rvert-\log\left\lvert y-\overline{p_1^k}^{-1}\right\rvert\right)dy\right\rvert\leq 2\sup_{z\in D^2}\frac{1}{\lvert z-\overline{p_k}^{-1}\rvert}
\end{align*}
and since $p_1^k\to 0$, the right hand side tends to zero.\\
In order to show (\ref{Equation: Limit of gradient of log a e}) when $p_1=0$ we have to check that
\begin{align*}
\frac{x-\overline{p_1^k}^{-1}}{\left\lvert x-\overline{p_1^k}^{-1} \right\rvert^2}\to 0\text{ a.e.},
\end{align*}
but this is clear, since $p_1^k\to 0$.
\end{proof}
From Claim 3 and Claim 4 we deduce that by Dominated Convergence
\begin{align*}
\lim_{k\to \infty}\int_{D^2}f(a_k)\lvert \nabla a_k\rvert^2=\int_{D^2}f(a)\lvert \nabla a\rvert^2.
\end{align*}
\end{proof}
\begin{rem}\label{Remark: d_i even for points on the boundary}
The proof of Lemma \ref{Lemma: Stability wrt points} (and in particular Claim 2) shows that when a topological singularity of degree $d$ approaches the boundary, in the limit it becomes a singularity of degree $2d$.
This justifies the fact that throughout this paper we require the degree of topological singularities lying on $\partial D^2$ to be even.
\end{rem}
\begin{oq}
Does the result of Lemma \ref{Lemma: Stability wrt points} remains true if we assume that $g_0$ lies in $W^{1,1}(\partial D^2)$? And is the renormalized Dirichlet Energy stable under perturbation of the boundary datum $g_0$ (in $W^{1,1}(\partial D^2)$ or $H^\frac{1}{2}(\partial D^2)$)?
\end{oq}

\section{Proof of Theorem I.1 and Theorem \ref{Proposition: A priori estimate on number of singularities}}
In this section we provide a proof of Theorem \ref{th-I.1} and Theorem \ref{Proposition: A priori estimate on number of singularities}.\\
Theorem \ref{Proposition: A priori estimate on number of singularities} follows directly from the a priori estimate on the number of topological singularities given by Lemma \ref{Lemma: A priori estimate on the number of singularities} and Corollary \ref{Corollary: Estimate on number of singularities with boundary also in H1/2}. The result of Theorem \ref{Proposition: A priori estimate on number of singularities} is then applied to prove Theorem \ref{th-I.1}.


\subsection{Proof of Theorem \ref{Proposition: A priori estimate on number of singularities}}
\begin{lem}[\textbf{A priori estimate on the number of singularities}]\label{Lemma: A priori estimate on the number of singularities}
Let $g_0\in W^{1,1}(\partial D^2, S^1)$.
Let $Q\in \mathbb{N}$ and for any $i\in \{1,...,Q\}$ let $p_i\in D^2$ and $d_i\in \mathbb{Z}$.
Assume that
\begin{align}\label{Equation: First assumption a priori estimate}
i\int_{\partial D^2}g_0^{-1}\partial_{\theta}g_0+2\pi\sum_{\substack{i=1}}^Qd_i=0.
\end{align}
Let $a\in W^{1,1}(D^2, \mathbb{R})$ and assume that for any $\phi\in C^\infty(\overline{D^2})$
\begin{align}\label{Equation: Second assumption a priori estimate}
\int_{D^2}\nabla a\nabla\phi =-i\int_{\partial D^2}\phi\, g_0^{-1}\partial_\theta g_0-2\pi\sum_{i=1}^Q d_i\phi(p_i).
\end{align}
Let
$$f:\mathbb{R}\to\mathbb{R},\quad x\mapsto \frac{e^{2x}}{(1+e^{2x})^2}.$$
Then
\begin{align}
\nonumber
\pi\sum_{\substack{i=1\\d_i>0}}^Q d_i\leq \frac{1}{2}\lVert\partial_\theta g_0\rVert_{L^1(\partial D^2)}+\int_{D^2}f(a)\lvert \nabla a \rvert^2 .
\end{align}
and thus
\begin{align*}
\sum_{i=1}^Q\lvert d_i\rvert\leq\frac{3}{2\pi}\lVert \partial_\theta g_0\rVert_{L^1}+\frac{2}{\pi}\int_{D^2}f(a)\lvert \nabla a\rvert^2.
\end{align*}
\end{lem}
\begin{proof}
Assume first that $g_0\in C^\infty(\partial D^2, S^1)$.\\
For any $t\in \mathbb{R}$ let
\begin{align*}
A_t:=\{x\in D^2, a(x)<t\}.
\end{align*}

\textbf{Claim 1:} for a.e. $t\in \mathbb{R}$
\begin{align*}
\int_{a^{-1}(t)}\partial_\nu a=i\int_{\partial D^2\cap A_t}g_0^{-1}\partial_\theta g_0+2\pi\sum_{\substack{i=1\\d_i>0}}^Q d_i,
\end{align*}
where $\nu$ is the outer normal vector of the set $A_t$.
\begin{proof}[Proof of Claim 1] Since $a\in C^\infty(\overline{D^2}\smallsetminus\{p_1,...,p_Q\})$, for a.e. $t\in \mathbb{R}$ the set $A_t$ is an open subset of $D^2$ such that $\partial A_t$ is piecewise smooth and does not contain any topological singularity.\\
As $a$ solves
\begin{align*}
\Delta a=2\pi\sum_{i=1}^Q d_i\text{ in }D^2,
\end{align*}
for any such $t$ there holds
\begin{align}\label{Equation: Divergence theorem on sublevel sets}
\int_{\partial A_t}\partial_{\nu}a=2\pi\sum_{\substack{i=1\\p_i\in A_t}}^Q d_i.
\end{align}
Now
\begin{align*}
\partial A_t =a^{-1}(t)\cup (\partial D^2\cap A_t)
\end{align*}
and
\begin{align*}
\partial_\nu a=-ig_0^{-1}\partial_\theta g_0\text{ on }\partial D^2.
\end{align*}
Moreover for any $t\in \mathbb{R}$ a point $p_i$ lies in $A_t$ if and only if $d_i>0$, since
\begin{align*}
a(x)=d_i \log\lvert x-p_i\rvert+\mathscr{O}(1)
\end{align*}
in a neighbourhood of $p_i$ (see Lemma \ref{Lemma: Explicit form for a}).
Thus (\ref{Equation: Divergence theorem on sublevel sets}) implies
\begin{align*}
\int_{a^{-1}(t)}\partial_\nu a=i\int_{\partial D^2\cap A_t}g_0^{-1}\partial_\theta g_0+2\pi\sum_{\substack{i=1\\d_i>0}}^Q d_i.
\end{align*}
\end{proof}
Since the derivative of $a$ vanishes along $a^{-1}(t)$, on $a^{-1}(t)$ there holds
$$\partial_{\nu}a=\lvert \nabla a\rvert.$$
Therefore Claim 1 implies that for a.e. $t\in \mathbb{R}$
\begin{align}\label{Equation: Estimate on number of positive singularities, on a level set}
2\pi\sum_{\substack{i=1\\ d_i>0\\p_i\in D^2}}^Q d_i\leq \lVert \partial_\theta g_0\rVert_{L^1(\partial D^2)}+\int_{a^{-1}(t)}\partial_{\nu}a.
\end{align}
Now since $a \in W^{1,1}$, by Theorem 11 in \cite{Hajlasz} there exists a representative of $a$ for which the co-area formula holds. For such a representative we have
\begin{align}\label{Equation: Coarea formula to bound number of singularities, general application, 2}
\int_{D^2}f(a)\lvert \nabla a \rvert^2 dx&=\int_{\mathbb{R}}\left(\int_{a^{-1}(t)}f(a(x))\lvert\nabla a(x)\rvert d\mathscr{H}^1(x)\right)dt\\
\nonumber
&=\int_{\mathbb{R}}f(t)\left(\int_{{a}^{-1}(t)}\partial_\nu a(x) d\mathscr{H}^1(x) \right)dt.
\end{align}
Observe that
\begin{align}
\nonumber
\int_{\mathbb{R}}f(t)dt=\int_{\mathbb{R}}\frac{e^{2t}}{(1+e^{2t})^2}dt=\frac{1}{2}\int_0^\infty\frac{1}{(1+x)^2}dx=\frac{1}{2}.
\end{align}
Therefore multiplying both sides of (\ref{Equation: Estimate on number of positive singularities, on a level set}) by $f(t)$ and integrating on $\mathbb{R}$ we obtain
$$\pi\sum_{\substack{i=1\\ d_i>0}}^Q d_i\leq \frac{1}{2}\lVert\partial_\theta g_0\rVert_{L^1(\partial D^2)}+\int_{\mathbb{R}}f(t)\left(\int_{a^{-1}(t)}\partial_{\nu}a\right) dt.$$
Thus by (\ref{Equation: Coarea formula to bound number of singularities, general application, 2})
$$\pi\sum_{\substack{i=1\\ d_i>0}}^Q d_i\leq \frac{1}{2}\lVert\partial_\theta g_0\rVert_{L^1(\partial D^2)}+\int_{D^2}f(a)\lvert \nabla a \rvert^2 dx$$
and combining (\ref{Equation: First assumption a priori estimate}) and (\ref{Equation: Coarea formula to bound number of singularities, general application, 2}) we obtain
\begin{align*}
\sum_{i=1}^Q\lvert d_i\rvert=& 2\sum_{\substack{i=1\\ d_i>0}}^Qd_i-\sum_{i=1}^Q d_i\leq  \frac{2}{\pi}\left(\frac{1}{2}\lVert\partial_\theta g_0\rVert_{L^1(\partial D^2)}+\int_{D^2}f(a)\lvert \nabla a \rvert^2 dx\right)+\frac{1}{2\pi}\lVert \partial_\theta g_0\rVert_{L^1}\\
=&\frac{3}{2\pi}\lVert \partial_\theta g_0\rVert_{L^1}+\frac{2}{\pi}\int_{D^2}f(a)\lvert \nabla a\rvert^2.
\end{align*}
This concludes the proof under the assumption that $g_0\in C^\infty(\partial D^2, S^1)$.\\
Next consider the case where $g_0$ is a generic element of $W^{1,1}(\partial D^2, S^1)$. Let
$$R:=\sup_{i\in \{1,...,Q\}}\lvert p_i\rvert<1.$$
Since $a$ is smooth in $D^2\smallsetminus\{p_1,...,p_Q\}$, arguing as above we see that for any $r\in (R, 1)$, for any $t\in \mathbb{R}$
\begin{align}\label{Equation: Explicit form for integral of gradient on level set, inside}
\int_{a^{-1}(t)\cap D^2_r}\lvert \nabla a\rvert=-\int_{a^{-1}(t)\cap D^2_r}\partial_{\nu}a =\int_{\partial D_r^2\cap A_t}\partial_\nu a+2\pi\sum_{\substack{i=1\\ d_i>0}}^Q d_i.
\end{align}
The following Claim implies that taking the limit $r\to 1^-$ in (\ref{Equation: Explicit form for integral of gradient on level set, inside}) we recover estimate (\ref{Equation: Estimate on number of positive singularities, on a level set}), therefore we can conclude as in the previous case.\\

\textbf{Claim 2:}
\begin{align*}
\lim_{r\to 1^-}\left\lvert\int_{\partial D_r^2\cap A_t}\partial_\nu a\right\rvert\leq \lVert \partial_\theta g\rVert_{L^1(\partial D^2)}.
\end{align*}
\begin{proof}[Proof of Claim 2] For any $r\in (R,1)$ let
\begin{align*}
\iota_r: \partial D^2\to \partial D^2_r,\quad x\mapsto rx.
\end{align*}
To prove the Claim it is enough to show that
\begin{align*}
\frac{1}{r}\partial_\nu a\big\vert_{\partial D^2_r}\circ\iota_r\to -ig_0^{-1}\partial_\theta g_0\text{ in }L^1(\partial D^2)
\end{align*}
as $r\to 1^-$ (where $\partial_\nu$ denotes the outer normal derivative on $\partial D^2_r$).
To see this let's write
\begin{align*}
a=\tilde{a}+\Phi
\end{align*}
as in Lemma \ref{Lemma: Explicit form for a}. Then
\begin{align*}
\partial_{\nu}\Phi\vert_{\partial D^2_r}\circ \iota_r\to \partial_{\nu}\Phi\vert_{\partial D^2}\text{ in }L^1(\partial D^2)
\end{align*}
as $r\to 1^-$.
Moreover since $\tilde{a}$ is harmonic we have, using polar coordinates,
\begin{align*}
\partial_{r} \tilde{a}(r,\theta)=\tilde{a}\vert_{\partial D^2}\ast \partial_r P_r(\theta)=\frac{1}{r}\tilde{a}\vert_{\partial D^2}\ast H\partial_\theta P_r(\theta)=\frac{1}{ r} H\left(\partial_\theta\tilde{a}\vert_{\partial D^2}\right)\ast P_r(\theta).
\end{align*}
where $P_r$ denotes the Poisson kernel and $H$ the Hilbert transform, and we made use of the identity
\begin{align*}
H\partial_\theta P_r(\theta)=r\partial_r P_r(\theta).
\end{align*}
Next we claim that
\begin{align}\label{Equation: Identifying Hpartialtheta tildea}
H\left(\partial_\theta\tilde{a}\vert_{\partial D^2}\right)=-ig_0^{-1}\partial_\theta g_0-\partial_\nu \Phi\vert_{\partial D^2}.
\end{align}
To see this let $\phi\in C^\infty(\partial D^2, \mathbb{R})$ and denote by $\tilde{\phi}$ its harmonic extension in $D^2$. Then
\begin{align*}
\langle H\partial_\theta \tilde{a},\phi\rangle=\int_{\partial D^2}\tilde{a}H\partial_\theta \phi=\int_{\partial D^2}\tilde{a}\partial_\nu \phi=\int_{D^2}\nabla \tilde{a}\nabla \tilde{\phi}=-\int_{\partial D^2}\phi \left(ig_0^{-1}\partial_\theta g_0+\partial_\nu \Phi\big\vert_{\partial D^2}\right).
\end{align*}
 In the last step we made use of assumption (\ref{Equation: Second assumption a priori estimate}) and Lemma \ref{Lemma: Explicit form for a}.
As $\phi$ was arbitrary we conclude that (\ref{Equation: Identifying Hpartialtheta tildea}) holds true.\\
Then in particular
\begin{align*}
\partial_\nu \tilde{a}\big\vert_{\partial D^2_r}\circ\iota_r=-\frac{1}{r}(ig_0\partial_\theta g_0+\partial_\nu \Phi)\ast P_r\to -ig_0\partial_\theta g_0-\partial_\nu \Phi\text{ in }L^1(\partial D^2)
\end{align*}
as $r\to 1^-$,
since $(P_r)_{r\in (0,1)}$ is a family of approximated identities.
\end{proof}
\end{proof}
From the previous result and Lemma \ref{Lemma: Stability wrt points} we deduce that the same estimate holds if we allow singular points to lie on $\partial D^2$ provided that $g_0\in W^{1,1}\cap H^\frac{1}{2}(\partial D^2)$:
\begin{cor}\label{Corollary: Estimate on number of singularities with boundary also in H1/2}
Let $g_0\in W^{1,1}\cap H^\frac{1}{2}(\partial D^2)$. Let $Q\in \mathbb{N}$ and for any $i\in \{1,...,Q\}$ let $p_i\in \overline{D^2}$ and $d_i\in \mathbb{Z}$.
Assume that
\begin{align}
\nonumber
i\int_{\partial D^2}g_0^{-1}\partial_{\theta}g_0+2\pi\sum_{\substack{i=1\\ p_i\in D^2}}^Qd_i+\pi\sum_{\substack{i=1\\ p_i\in \partial D^2}}^Q d_i=0.
\end{align}
For any $i\in \{1,...,Q\}$ assume that whenever $p_i\in \partial D^2$ $d_i$ is even.\\
Let $a\in W^{1,1}(D^2, \mathbb{R})$ and assume that for any $\phi\in C^\infty(\overline{D^2})$
\begin{align*}
\int_{D^2}\nabla a\nabla\phi =-i\int_{S^1}\phi g_0^{-1}\partial_\theta g_0-2\pi\sum_{\substack{i=1\\ p_i\in D^2}}^Q d_i\phi(p_i)-\pi\sum_{\substack{i=1\\ p_i\in \partial D^2}}^Q d_i\phi(p_i).
\end{align*}
Then
\begin{align}\label{Equation: Estimate number of topological singularities W11H12}
\pi\sum_{\substack{i=1\\ d_i>0\\p_i\in D^2}}^Q d_i+\frac{1}{2}\pi\sum_{\substack{i=1\\ d_i>0\\ p_i\in \partial D^2}}^Q d_i\leq \frac{1}{2}\lVert\partial_\theta g_0\rVert_{L^1(\partial D^2)}+\int_{D^2}f(a)\lvert \nabla a \rvert^2 dx
\end{align}
and thus
\begin{align}\label{Equation: Estimate number of topological singularities W11H12,2}
\sum_{i=1}^Q\lvert d_i\rvert\leq\frac{3}{\pi}\lVert \partial_\theta g_0\rVert_{L^1}+\frac{4}{\pi}\int_{D^2}f(a)\lvert \nabla a\rvert^2.
\end{align}
\end{cor}
\begin{proof}
Assume that $p_1,...,p_{Q'}$ lie on $\partial D^2$ while $p_{Q'+1},..., p_Q$ lie on $D^2$.
For any $i\in \{1,..., Q'\}$ let $\tilde{d_i}=\frac{d_i}{2}$.
Let $\varepsilon>0$.
According to Lemma \ref{Lemma: Stability wrt points}, for any $i\in \{1,...,Q'\}$ we can choose a point $p_i^\varepsilon\in D^2$ such that the corresponding function $a_\varepsilon$ (with degrees $\tilde{d_1},...,\tilde{d_{Q'}},d_{{Q'}+1},...,d_Q$ and same outer normal derivative as $a$) satisfies
\begin{align*}
\left\lvert\int_{D^2}f(a_\varepsilon)\lvert\nabla a_\varepsilon\rvert^2-\int_{D^2}f(a)\lvert\nabla a\rvert^2\right\rvert<\varepsilon.
\end{align*}
Now since all all the topological singularities of $a_\varepsilon$ lie in $D^2$, Lemma \ref{Lemma: A priori estimate on the number of singularities} implies
\begin{align*}
\pi\sum_{\substack{i=1\\ d_i>0\\p_i\in D^2}}^Q d_i+\frac{1}{2}\pi\sum_{\substack{i=1\\ d_i>0\\ p_i\in \partial D^2}}^Q d_i=&\pi\sum_{\substack{i=Q'+1}}^Q d_i+\pi\sum_{\substack{i=1}}^{Q'} \tilde{d}_i\\
\leq& \frac{1}{2}\lVert\partial_\theta g_0\rVert_{L^1(\partial D^2)}+\int_{D^2}f(a_\varepsilon)\lvert \nabla a_\varepsilon \rvert^2 dx\\ \leq& \frac{1}{2}\lVert\partial_\theta g_0\rVert_{L^1(\partial D^2)}+\int_{D^2}f(a)\lvert \nabla a \rvert^2 dx+\varepsilon.
\end{align*}
Letting $\varepsilon$ tend to $0$ we obtain (\ref{Equation: Estimate number of topological singularities W11H12}), from which (\ref{Equation: Estimate number of topological singularities W11H12,2}) can be decuced as in the proof of Lemma \ref{Lemma: A priori estimate on the number of singularities}.
\end{proof}
\begin{oq}
Although we assume $g_0$ to lie in $H^\frac{1}{2}(\partial D^2)$ in Corollary \ref{Corollary: Estimate on number of singularities with boundary also in H1/2}, the $H^\frac{1}{2}$-norm of $g_0$ does not appear on the right hand side of estimate (\ref{Equation: Estimate number of topological singularities W11H12}). It is natural to wonder if the result remains true if we only assume $g_0$ to lie in $W^{1,1}(\partial D^2)$, or if we could substitute the $W^{1,1}$-norm with the $H^\frac{1}{2}$-norm of $g_0$ in (\ref{Equation: Estimate number of topological singularities W11H12}) (assuming only $g_0\in H^\frac{1}{2}(\partial D^2)$).
\end{oq}

From Lemma \ref{Lemma: A priori estimate on the number of singularities} we deduce also the following estimate.
\begin{lem}\label{Lemma: Estimate from estimate on number of topological singularities}
Let $a$ be as in Lemma \ref{Lemma: A priori estimate on the number of singularities} or as in Corollary \ref{Corollary: Estimate on number of singularities with boundary also in H1/2}. There exists a constant $C$ independent from $a$ such that
\begin{align}
\lVert \nabla a\rVert_{L^{2,\infty}(D^2)}\leq C\left(\lVert g_0\rVert_{W^{1,1}(\partial D^2)}+\lVert \nabla u_g\rVert_{L^2(D^2)}\right)
\end{align}
\end{lem}
\begin{proof}
Let $X\in L^{2,1}(D^2)$ be a vector field. Then by Lemma \ref{Lemma: Hodge decomposition in L2,1} there exist functions $\eta\in W^{1,(2,1)}(D^2)$, $\xi\in W^{1, (2,1)}_0(D^2)$ such that
$$\lVert\eta\rVert_{W^{1, (2,1)}(D^2)}\leq C\lVert X\rVert_{L^{2,1}(D^2)}$$
and
$$X=\nabla \eta+\nabla^\perp \xi.$$
Recall that $W^{1, (2,1)}(D^2)\subset C^0(\overline{D^2})$ and the embedding is continuous.\\
Then there holds
\begin{align*}
\int_{D^2} X\nabla a&=\int_{D^2} \nabla \eta\nabla a+\int_{D^2}\nabla^\perp \xi \nabla a=\int_{D^2}\nabla\eta\nabla a\\&=-i\int_{\partial D^2} \eta g_0^{-1}\partial_\theta g_0-2\pi\sum_{\substack{i=1\\ p_i\in D^2}}^Qd_i\eta(p_i)-\pi\sum_{\substack{i=1\\ p_i\in \partial D^2}}^Qd_i\eta(p_i).
\end{align*}
Here we used the fact that since $\xi$ has vanishing trace on $\partial D^2$, integrating by parts
we get
$$\int_{D^2}\nabla^\perp \xi \nabla a=0.$$
Moreover we know from Lemma \ref{Lemma: A priori estimate on the number of singularities} (or Corollary \ref{Corollary: Estimate on number of singularities with boundary also in H1/2}) that
\begin{align*}
\sum_{i=1}^Q\lvert d_i\rvert\leq C\left(\int_{D^2}f(a)\lvert\nabla a\rvert^2+\lVert g_0\rVert_{W^{1,1}(\partial D^2)}\right)
\end{align*}
Therefore there holds
\begin{align}
\nonumber
\int_{D^2}X\nabla a\leq & C\left(\lVert g_0^{-1}\partial_\theta g_0\rVert_{L^1(\partial D^2)}+\int_{D^2}f(a)\lvert\nabla a\rvert^2dx+\lVert g_0\rVert_{W^{1,1}(\partial D^2)}\right)\lVert \eta\rVert_{L^{\infty}(D^2)}\\
\nonumber
\leq & C\left(\mathcal{E}(u_g)+\lVert g_0\rVert_{W^{1,1}(\partial D^2)}\right)\lVert X\rVert_{L^{2,1}(D^2)}.
\end{align}
As the above estimate holds true for any $X\in L^{2,1}(D^2)$ the statement follows.
\end{proof}

\subsection{Proof of Theorem I.1}
We are now ready to prove Theorem I.1.
\begin{proof}[Proof of Theorem I.1]
\begin{enumerate}

\item[a)]
First we observe that by condition (\ref{I.15}) the sequence $(b_{g_k})_{k\in \mathbb{N}}$ is bounded in $W^{1,2}(D^2)$. Therefore there exists a function $b\in W^{1,2}(D^2)$ such that
\begin{align*}
b_{g_k}\rightharpoonup b\text{ weakly in }W^{1,2}(D^2)\text{ and a.e.}
\end{align*}
up to a subsequence.\\

\textbf{Claim 1:} There is a subsequence of $(g_k)_{k\in \mathbb{N}}$, say indexed by $\Lambda\subset \mathbb{N}$, a function $g\in W^{1, (2,\infty)}_{loc}\cap W^{1,p}(D^2,S^1)$, a zero average function $a\in W^{1, (2,\infty)}_{loc}\cap W^{1,1}(D^2,\mathbb{R})$ and $b\in W^{1, 2}_0(D^2,\mathbb{R})$  such that
\begin{align*}
a_{g_k}\to a\text{ a.e. and }\nabla a_{g_k}\rightharpoonup \nabla a \text{ weakly-}* \text{ in }L^{2,\infty}_{loc}(D^2),
\end{align*}
\begin{align*}
g_k\to g\text{ a.e. and }\nabla g_k\rightharpoonup \nabla g \text{ weakly-}*\text{ in }L^{2,\infty}_{loc}(D^2)
\end{align*}
along $\Lambda$,
\begin{align*}
-ig^{-1}\nabla g=\,\nabla^\perp a+\nabla b
\end{align*}
and $g$ has isolated topological singularities in $D^2$.
\begin{proof}[Proof of Claim 1]
First observe that by Lemma \ref{Lemma: Pseudo Courant Lebesgue} for any $n\in \mathbb{N}$ there exists $r_n\in (1-\frac{1}{n},1)$ and a subsequence indexed by $\Lambda_n\subset\mathbb{N}$ such that for any $k\in \Lambda_n$ $g_k$ has no topological singularities on $\partial D^2_{r_n}$ and
\begin{align*}
\sup_{k\in \Lambda_n}\int_{\partial D^2_{r_n}}\lvert ig_k^{-1}\nabla g_k+\nabla b_{g_k}\rvert<\infty.
\end{align*}
Thus by Lemma \ref{Lemma: A priori estimate on the number of singularities} (applied to the sequence of functions $(e^{ib_{g_k}}g_k)_{k\in \Lambda_n}$) the number of topological singularities of $a_k$ in $D_{r_n}$ (counted with multiplicities) is uniformly bounded for all $k\in \Lambda_n$.
Now let $n\in \mathbb{N}$. By Lemma \ref{Lemma: Estimate from estimate on number of topological singularities} (applied to $D_{r_n}$) the sequence $(\nabla a_{g_k})_{k\in \Lambda_n}$ is bounded in $L^{2,\infty}(D^2_{r_n})$.
By Poincar\'e Lemma and Banach-Alaoglu Theorem\footnote{Since $L^{2,\infty}(D^2_{r_n})=\left(L^{2,1}(D^2_{r_n})\right)^\ast$ and $L^{2,1}(D^2_{r_n})$ is separable, Banach-Alaoglu holds for $L^{2,\infty}(D^2_{r_n})$.} the sequence
\begin{align*}
\left(a_{g_k}-\fint_{D_{r_n}}a_{g_k}\right)_{k\in \Lambda_n}
\end{align*}
has a subsequence converging weakly in $W^{1,(2,\infty)}(D_{r_n})$ to a function $a_{n}\in W^{1,(2,\infty)}(D_{r_n})$.\\
Iterating this argument for any $n\in \mathbb{N}$ and extracting a diagonal subsequence we find for any $n\in \mathbb{N}$ a function $a_{n}\in W^{1,(2,\infty)}(D^2_{r_n})$ such that for $n,m\in \mathbb{N}$, $n\leq m$ the function $a_{n}-a_{m}$ is constant in $D_{r_n}\cap D_{r_m}$.\\
All these functions can be glued together (subtracting a constant whenever necessary) to obtain a function $a\in W^{1,(2,\infty)}_{loc}(D^2)$ (uniquely defined up to an additive constant) such that
\begin{align*}
\nabla a_{g_k}\rightharpoonup\nabla a\text{ weakly-}*\text{ in }L^{2,\infty}_{loc}(D^2)
\end{align*}
up to a subsequence.
Moreover since the sequence $(a_{g_k})_{k\in \mathbb{N}}$ is bounded in $W^{1,p'}(D^2)$ (with $p'=\min\{2,p\}$), taking a further subsequence if necessary we may assume that the convergence also takes place weakly in $W^{1,p'}(D^2)$ a.e.. Then $a\in W^{1,1}(D^2)$ and the weak $W^{1,p'}$-convergence determines the additive constant in the definition of $a$. In particular $a$ has average zero.\\
For any $n\in \mathbb{N}$ Corollary \ref{Lemma: Explicit form for g} (applied to $D_{r_n}^2$) implies that the sequence $(g_k)_{k\in \Lambda_n}$ is also bounded in $W^{1, (2,\infty)}(D_{r_n}^2)$, since for any $k\in \Lambda_n$
\begin{align*}
\lvert \nabla g_k\rvert\leq \lvert \nabla \varphi_k\rvert+\sum_{\substack{i=1\\ p_i^k\in D_{r_n}^2}}^{Q_k}\frac{\lvert d_i^k\rvert}{\lvert z-p_i^k\rvert}\leq\lvert \nabla^\perp a_{g_k}\rvert+\lvert \nabla b_{g_k}\rvert+\sum_{\substack{i=1\\ p_i^k\in D_{r_n}^2}}^{Q_k}\frac{\lvert d_i^k\rvert}{\lvert z-p_i^k\rvert},
\end{align*}
where $\varphi_k$ was defined in Corollary \ref{Lemma: Explicit form for g}, and $\displaystyle \sum_{\substack{i=1\\ p_i^k\in D_{r_n}^2}}^{Q_k}\lvert d_i^k\rvert$ is uniformly bounded for any $k\in \Lambda_n$. Therefore applying Banach-Alaoglu Theorem and Rellich-Kondrachov Theorem to the sequence $(g_k\vert_{D^2_{r_n}})_{k\in \Lambda_n}$ for any $n\in \mathbb{N}$
and extracting a diagonal subsequence we find a function $g\in W^{1,(2,\infty)}_{loc}(D^2, S^1)$ such that
\begin{align*}
g_k\to g\text{ a.e. and }\nabla g_k\rightharpoonup\nabla g\text{ weakly-}*\text{ in }L^{(2,\infty)}_{loc}(D^2)
\end{align*}
up to a subsequence. Moreover since the sequence $(g_k)_{k\in \mathbb{N}}$ is bounded in $W^{1,p}(D^2)$, taking a further subsequence if necessary we may assume that the convergence also takes place weakly in $W^{1,p}(D^2)$. In particular $g\in W^{1,p}(D^2)$.\\
It follows that
\begin{align*}
g_k^{-1}\nabla g_k\to g^{-1}\nabla g \text{ in } \mathcal{D}'(D^2)
\end{align*}
up to a subsequence. Therefore, since for any $k\in \mathbb{N}$
\begin{align*}
-ig_k^{-1}\nabla g_k=\nabla^\perp a_{g_k}+\nabla b_{g_k},
\end{align*}
we conclude that
\begin{align*}
-ig^{-1}\nabla g=\nabla^\perp a+\nabla b.
\end{align*}
We still need to check that $g$ has isolated singularities in $D^2$.
To this end let's fix $n\in \mathbb{N}$.
Observe that for any $k\in \Lambda_n$ $a_{g_k}$ satisfies
$$\Delta a_{g_k}=2\pi\sum_{\substack{i=1\\ p_i^k\in D_{r_n}^2}}^{Q_{k}}d_i^k\delta_{p_i^k}\text{ in }D^2_{r_n}$$
for some integers $Q_k$, $d_1^k,...,d_{Q_k}^k$ and points $p_1,...,p_{Q_k}$ in $D_{r_n}^2$, where $\displaystyle \sum_{\substack{i=1\\ p_i^k\in D_{r_n}^2}}^{Q_k}\lvert d_i^k\rvert$ is uniformly bounded for any $k\in \Lambda_n$.\\
Upon passing to a subsequence we may assume that the integers $Q_k$ and $d_i^k$ do not depend on $k$ (therefore we will drop the $k$ in the notation) and that for any $i\in \{1,...,Q\}$
$$p_i^k\to p_i$$
for some points $p_1,...,p_Q\in \overline{D^2_{r_n}}$.
Then for any $\varphi\in C_c^\infty(D^2_{r_n})$
$$\int_{D^2_{r_n}}\nabla \varphi\nabla a=\lim_{k\to \infty}\int_{D^2_{r_n}}\nabla \varphi\nabla a_{g_k}=2\pi\sum_{i=1}^Qd_i\phi(p_i),$$
therefore
$$\operatorname{div}(ig^{-1}\nabla^\perp g)=\Delta a=2\pi\sum_{i=1}^Q d_i\delta_{p_i}\text{ in }D^2_{r_n}.$$
Thus $g$ has finitely many topological singularities in $D^2_{r_n}$ and since this is true for any $n\in \mathbb{N}$ we conclude that $g$ has isolated topological singularities in $D^2$.
\end{proof}

\textbf{Claim 2:} There exists a subsequence of $(g_k)_{k\in \Lambda}$, say indexed by $\tilde{\Lambda}$, such that
$$\mathcal{E}(g)\leq \liminf_{\substack{k\to \infty\\ k\in \tilde{\Lambda}}}\mathcal{E}(g_k).$$
\begin{proof}[Proof of Claim 2]

Observe that by condition (\ref{th-I.1}) there is a subsequence, say indexed by $\tilde{\Lambda}\subset\Lambda$, and a map $u\in W^{1,2}(D^2, S^2)$ such that
\begin{align*}
u_{g_k}\rightharpoonup u\text{ weakly in }W^{1,2}(D^2, S^2)\text{ and a.e.}
\end{align*}
along $\tilde{\Lambda}$.
Then
$$\int_{D^2}\lvert \nabla u\rvert^2\leq \liminf_{\substack{k\to \infty\\ k\in \tilde{\Lambda}}}\int_{D^2}\lvert \nabla u_{g_k}\rvert^2.$$
Since for any $k\in \mathbb{N}$
$$u_{g_k}=\pi^{-1}(g_k e^{a_k}),$$
the pointwise convergence of $(g_k)_{k\in \tilde{\Lambda}}$, $(a_k)_{k\in \tilde{\Lambda}}$ and $(u_{g_k})_{k\in \tilde{\Lambda}}$ implies that
$$u=\pi^{-1}(g e^{a}).$$
Moreover replacing $\tilde{\Lambda}$ with a subsequence if necessary we have
$$\int_{D^2}\lvert \nabla b\rvert^2\leq\liminf_{\substack{k\to\infty\\ k\in \tilde{\Lambda}}}\int_{D^2}\lvert \nabla b_{g_k}\rvert^2.$$
Therefore
$$\mathcal{E}(g)=\frac{1}{4}\int_{D^2}\lvert \nabla u\rvert^2
+\frac{1}{4}\int_{D^2}\lvert\nabla b\rvert^2
\leq \liminf_{\substack{k\to \infty\\ k\in \tilde{\Lambda}}}\mathcal{E}(g_k).$$
\end{proof}

\item[b)] By Lemma \ref{Lemma: Estimate from estimate on number of topological singularities} (applied to the case $g_0\in W^{1,1}\cap H^\frac{1}{2}(\partial D^2)$) the sequence $(\nabla a_{g_k})_{g_k}$ is bounded in $L^{2,\infty}(D^2)$.
By Corollary \ref{Lemma: Explicit form for g} the sequence $(g_k)_{k\in \mathbb{N}}$ is bounded in $W^{1,(2,\infty)}(D^2)$.
Therefore by Banach-Alaoglu Theorem and Rellich-Kondrachov Theorem we can find a zero average function $a\in W^{1, (2,\infty)}(D^2)$, $b\in W^{1,2}_0(D^2)$, $g\in W^{1,(2,\infty)}(D^2, S^1)$ such that
\begin{align*}
a_k\to a\text{ a.e. and }\nabla a_k\rightharpoonup \nabla a\text{ weakly-}\ast\text{ in }L^{2,\infty}(D^2),
\end{align*}
\begin{align*}
b_k\rightharpoonup b\text{ weakly in }W^{1,2}(D^2)\text{  and a.e.},
\end{align*}
\begin{align*}
g_k\to g\text{ a.e. and }\nabla g_k\rightharpoonup \nabla g\text{ weakly-}\ast\text{ in }L^{2,\infty}(D^2)
\end{align*}
along a subsequence and
\begin{align}\label{Equation: Hodge decomposition in the limit}
-ig^{-1}\nabla g=\nabla^\perp a+\nabla b\text{ in }D^2.
\end{align}
Moreover following the argument of Claim 2 in part a) we see that
\begin{align*}
\mathcal{E}(u_g)\leq \liminf_{\substack{k\to\infty\\ k\in \Lambda}}\mathcal{E}(u_{g_k})
\end{align*}
where $\Lambda\subset \mathbb{N}$ is the index set of a further subsequence.\\

\textbf{Claim 1:} $g$ has finitely many topological singularities in $\overline{D^2}$.
\begin{proof}[Proof of Claim 1]
For any $k\in \mathbb{N}$ denote $p_1^k,...,p_{Q_k}^k$, $d_1^k,...,d_{Q_k}^k$ the topological singularities of $g_k$ and their degrees.
By Corollary \ref{Corollary: Estimate on number of singularities with boundary also in H1/2} the number of topological singularities of $g_k$ and their degrees are uniformly bounded for all $k\in \mathbb{N}$.
Therefore there exists a subsequence, say indexed by $\Lambda'\subset\Lambda$, $Q\in \mathbb{N}$ points $p_1,...,p_Q\in \overline{D^2}$ and degrees $d_1,...,d_Q\in \mathbb{Z}$ such that for any $i\in \{1,...,Q\}$
\begin{align*}
d_i^k=d_i\quad\forall k\in \Lambda'
\end{align*}
and
\begin{align*}
\lim_{\substack{k\to\infty\\k\in \Lambda'}}p_i^k=p_i.
\end{align*}
Thus for any $\phi\in C^\infty(\overline{D^2})$ there holds
$$\int_{D^2}\nabla \phi\nabla a=\lim_{\substack{k\to\infty\\k\in \Lambda'}}\int_{D^2}\nabla \phi\nabla a_n=-i\int_{\partial D^2}\phi g_0^{-1}\partial_\theta g_0-2\pi\sum_{\substack{i=1\\ p_i\in D^2}}^Q d_i\phi(p_i)-\pi\sum_{\substack{i=1\\ p_i\in \partial D^2}}^Q d_i\phi(p_i),$$
so $g$ has finitely many topological singularities in $\overline{D^2}$.
\end{proof}
\end{enumerate}
\end{proof}

\section{Proof of Theorem I.3}
In this section we prove Theorem I.3. We divide the proof in two steps, corresponding to the two following Lemmas.
\begin{lem}\label{Lemma: Equivalent conditions g S^1 harmonic}
Let $g\in W^{1,1}(D^2, S^1)$ be as in Definition \ref{df-I.1}.
Assume that
\begin{align*}
\mathcal{E}(g)<\infty.
\end{align*}
Then the following conditions are equivalent:
\begin{enumerate}
\item $g$ is an $S^1$-harmonic map,
\item $b_g\equiv0$ in $D^2$.
\end{enumerate}
If $g$ has isolated topological singularities, either of the above condition is equivalent to
\begin{enumerate}
\item[3.] $u_g$ is conformal.
\end{enumerate}
\end{lem}

\begin{proof}
Assume first that $g$ is $S^1$-harmonic. We claim that $b_g\equiv0$.\\
By assumption
$$\operatorname{div}(g^{-1}\nabla g)=0.$$
Plugging in the decomposition (\ref{Equation: Definition of ag}) we obtain $$\operatorname{div}(\nabla^\perp a_g+\nabla b_g)=0.$$
Now since
$$\operatorname{div}(\nabla^\perp a_g)=0,$$
$b_g$ solves
\begin{align*}
\begin{cases}
\Delta b_g=0&\text{ in }D^2\\[5mm]
b_g=0 &\text{ on }\partial D^2.
\end{cases}
\end{align*}
Since $b_g\in H^1(D^2)$ we conclude that $b_g\equiv 0$ in $D^2$.\\
We now show the converse: assume that $b_g\equiv 0$. Then
$$g^{-1}\nabla g=i\nabla^\perp a_g.$$
Therefore
$$-i\operatorname{div}(g^{-1}\nabla g)=\operatorname{div}(\nabla^\perp a_g)=0.$$
We conclude that $g$ is an $S^1$-harmonic map.\\
Next assume that $g$ has isolated topological singularities. We will show that $b_g\equiv0$ if and only if $u_g$ is conformal.\\
Assume first that $b_g\equiv0$.
Then
$$
g^{-1}\nabla g=i\nabla^\perp a_g.
$$
Observe that
\begin{align}\label{Equation: Explicit expression for the gradient of ug}
\nabla u_g=D\pi^{-1}(e^{a_g}g)e^{a_g}(\nabla g+g\nabla a_g)
\end{align}
Then
\begin{align*}
\nabla u_g =D\pi^{-1}(e^{a_g}g)e^{a_g}g(i\nabla^\perp a_g+\nabla a_g).
\end{align*}
As $\pi$ is conformal and
$$
\nabla a_g(x)+i\nabla^\perp a_g(x)
$$
defines a conformal map at any point $x\in D^2$ away from the singular points of $a_g$ we conclude that $u_g$ is conformal away from the topological singularities of $g$.\\
Now if we consider separately the preimages of $S^2\smallsetminus\{S\}$ and $S^2\smallsetminus\{N\}$ (where $S$ and $N$ denote the south and the north pole of $S^2$ respectively) and look at the composition $\pi_i\circ u_g$ for the corresponding stereographic projection $\pi_i$ we notice that the singularities of $u_g$ in $D^2$, corresponding to the topological singularities of $g$ are removable. We conclude that $u_g$ is conformal.\\
Conversely, if we assume that $u_g$ is conformal computation (\ref{Equation: Explicit expression for the gradient of ug}) implies that
$$g^{-1}\nabla g+\nabla a_g$$
is conformal a.e., therefore
\begin{align*}
-ig^{-1}\nabla g=\nabla^\perp a_g
\end{align*}
and thus $b_g\equiv 0$.
\end{proof}

\begin{lem}\label{Lemma: Euler Lagrange Equation for smooth variations in the target}
Let $g\in W^{1,p}(D^2, S^1)$ for some $p>1$ with finite renormalized Dirichlet Energy.
Let
\begin{align}
\nonumber
f:\mathbb{R}\to\mathbb{R}, \quad x\mapsto \frac{e^{2x}}{(1+e^{2x})^2}.
\end{align}
Then 
$g$ is a critical point of the renormalized Dirichlet Energy for smooth variations in the target, that is
\begin{align}\label{Equation: Definition of critical for smooth variations in the target}
\forall \psi\in C^\infty_c(D^2, \mathbb{R})\quad \frac{d}{dt}\bigg\vert_{t=0}\mathcal{E}(g e^{it\psi})=0
\end{align}
if and only if $b_g\equiv 0$ in $D^2$.
\end{lem}
\begin{proof}
Let $\psi\in C^\infty_c(D^2, \mathbb{R})$. For any $t>0$ let
$$g_t=g e^{it\psi}.$$
Then
$$\nabla g_t= \nabla g\, e^{it\psi}+i tg\, e^{it\psi}\nabla\psi.$$
Therefore
$$g_t^{-1}\nabla g_t=g^{-1}\nabla g_t+it \nabla \psi=i \nabla^\perp a_g+i \nabla b_g+it \nabla\psi.$$
Thus it follows from (\ref{computations explicit form energy, 1}) that
\begin{align}\label{Equation: Computation EL}
\mathcal{E}(g_t)-\mathcal{E}(g)=&\int_{D^2}f(a_g)\left(\lvert \nabla g_t\rvert^2-\lvert \nabla g\rvert^2\right)+\frac{1}{4}\int_{D^2}\left(\lvert \nabla b_g+t\nabla \psi\rvert^2-\lvert \nabla b_g\rvert^2\right)\\
\nonumber
=&2t\int_{D^2}f(a_g)\left(<-ig^{-1}\nabla g, \nabla \psi>+\frac{1}{4}<\nabla b_g, \nabla \psi>\right)\\
\nonumber
&+t^2\int_{D^2}\left(f(a_g)\lvert \nabla \psi\rvert^2+\frac{1}{4}\lvert\nabla \psi\rvert^2\right)
\end{align}
and we get
$$\frac{d}{dt}\bigg\vert_{t=0}\mathcal{E}(g_t)=2\int_{D^2}f(a_g)<-ig^{-1}\nabla g, \nabla\psi>+\frac{1}{4}<\nabla b_g, \nabla\psi>.$$
Therefore $g$ is a critical point of $\mathcal{E}$ (with respect to the variations introduced in (\ref{Equation: Definition of critical for smooth variations in the target})) if and only if
\begin{align}\label{Equation: Cauchy problem for bg critical}
\text{div}\left(-i\,f(a_g)g^{-1}\nabla g+\frac{1}{4}\nabla b_g\right)=0.
\end{align}
%
%
Plugging in the Hodge decomposition (\ref{Equation: Definition of ag}) in (\ref{Equation: Cauchy problem for bg critical}) we obtain
\begin{align*}
\operatorname{div}\left(f(a_g)(\nabla^\perp a_g+\nabla b_g)+\frac{1}{4}\nabla b_g\right)=0.
\end{align*}
Recall that
$$\text{div}\left(f(a_g)\nabla^\perp a_g\right)=0$$
as we saw in (\ref{Equation: div(f(a)nablaa)=0}).\\
Therefore we can rewrite (\ref{Equation: Cauchy problem for bg critical}) as
\begin{align}\label{Equation: Second Cauchy problem for bg critical}
\text{div}\left(f(a_g)\nabla b_g+\frac{1}{4}\nabla b_g\right)=0.
\end{align}
We claim that the only solution in $W^{1,2}_0(D^2, \mathbb{R})$ of (\ref{Equation: Second Cauchy problem for bg critical}) is $b_g\equiv 0$.\\
In fact the equation (\ref{Equation: Second Cauchy problem for bg critical}) is the Euler-Lagrange equation of the energy
$$E: W^{1,2}_0(D^2, \mathbb{R})\to \mathbb{R}, \quad h\mapsto \int_{D^2}\left(f(a_g)+\frac{1}{4}\right)\lvert \nabla h\rvert^2.$$
As the energy $E$ is strictly convex, it has a single critical point, which has to be $b_g\equiv 0$.\\
We conclude that $b_g$ is a solution of (\ref{Equation: Cauchy problem for bg critical}) if and only if $b_g\equiv 0$.\\
\end{proof}

\begin{rem}\label{Remark: general variations}
One could also consider variations of the type
$$g_t=\frac{g+t\phi}{\lvert g+t\phi\rvert},$$
where $\phi\in C^\infty_c(D^2, \mathbb{R}^2)$.
Nevertheless there are $S^1$-harmonic maps $g$ with finite energy $\mathcal{E}$ and maps $\phi\in C^\infty_c(D^2, \mathbb{R}^2)$, for which the energy $\mathcal{E}(g_t)$ is not finite for any $t\neq 0$.\\
For instance let $g(re^{i\theta})=e^{i\theta}$, $e_1$ the first basis vector of $\mathbb{R}^2$ and $\eta\in C^\infty_c((-1,1))$ such that $\eta\equiv 1$ on $(-\frac{1}{2},\frac{1}{2})$.
Set
\begin{align*}
g_t(re^{i\theta})=\frac{e^{i\theta}+t\eta(r) e_1}{\lvert e^{i\theta}+t\eta(r) e_1\rvert}.
\end{align*}
For $t$ sufficiently small there holds $g_t=ge^{i\psi_t}$, where
\begin{align*}
\psi_t(re^{i\theta})=-\arctan\left(\frac{t\eta(r)\sin\theta}{1+t\eta(r)\cos\theta}\right).
\end{align*}
Therefore
\begin{align*}
g_t^{-1}\nabla g_t=g^{-1}\nabla g+i\nabla\psi_t.
\end{align*}
%
Since for our choice of $g$ we have $b_g\equiv 0$, there holds $b_{g_t}=\psi_t$. But $\psi_t\notin L^2(D^2)$ for $t\neq 0$ (as $\psi_t$ only depends on $\theta$ in a neighbourhood of zero), therefore $\mathcal{E}(g_t)=\infty$ whenever $t\neq 0$.\\
By requiring $g$ to have isolated topological singularities in $D^2$ and considering variations as above with $\phi$ supported away from the topological singularities of $g$ one can obtain a result analogous to Lemma \ref{Lemma: Euler Lagrange Equation for smooth variations in the target}.
\end{rem}

\section{Applications}
\reset
\subsection{The Lagrangian-Willmore Energy}
\subsubsection{The Hamiltonian stationary condition and Schoen-Wolfson isolated singularities}\label{The hamiltonian stationary condition and Schoen-Wolfson isolated singularities}
Let 
\[
\om:=dx_1\wedge dy_1+dx_2\wedge dy_2
\]
be the standard symplectic form in ${\R}^4$ compatible with the standard complex structure $J_0$ such that $J_0\,\p_{x_k}=\p_{y_k}$ . We shall denote $dz_k=dx_k+i\,dy_k$. An immersion $G$ from a surface $\Sigma$ into $({\R}^4,\om)$
is called Lagrangian if 
\[
G^\ast\om=0\ .
\]
This condition is equivalent to the fact that $J_0$ realizes an isometry between the tangent plane to the immersion and the normal plane. A short computation gives the existence of
a map ${\frak g}$ from $\Sigma$ into $S^1$ such that
\[
G^\ast dz_1\wedge dz_2 = {\frak g}\ dvol_G
\]
where $dvol_G$ denotes the volume form associated to the induced metric $G^\ast g_{{\R}^4}$. The map ${\frak g}$ is called the {\it Lagrangian angle function}.

Consider conformal coordinates  for $G$ with respect to $G^\ast g_{{\R}^4}$, with $G^\ast g_{{\R}^4}= e^{2\la}\ dx^2$; then a direct computation gives that
\[
i\,\Delta G= {i}^{-1}\,{\frak g}^{-1}\nabla{\frak g}\cdot\nabla G\quad\Longleftrightarrow\quad\mbox{div}\lf({\frak g}\nabla G\rg)=0.
\]
We also deduce the following expression for the mean curvature vector
\be
\label{I.20}
\vec{H}_G:=2^{-1}\,e^{-2\la}\,\Delta G =-2^{-1}\,e^{-2\la}\,{\frak g}^{-1}\nabla{\frak g}\cdot\nabla G.
\ee
The vector-fields preserving the Lagrangian condition infinitesimally in the ambient space ${\R}^4$ are called {\it Hamiltonian vector fields} and are of the form
\[
X=J_0\,\nabla^{{\R}^4} f=(-\p_{y_1} f,\p_{x_1} f,-\p_{y_2} f,\p_{x_2}f),
\]
where $f$ is an arbitrary function.  Thus the condition for being a critical point of the area under local perturbations preserving the Lagrangian condition (the so called {\it Hamiltonian stationary condition} introduced originally by Oh \cite{Oh1})
is given by
\be
\label{I.22}
\forall f\in C^\infty_0({\R}^4)\quad 0=\int_\Sigma J_0\nabla^{{\R}^4} f(G)\cdot \vec{H}_G\ dvol_G= \int_\Sigma <G^\ast df, i^{-1}\,{\frak g}^{-1}\,d{\frak g}>_{g_G}\ dvol_G.
\ee
Assume that $G$ is a smooth local immersion. Then, locally, every differential $d\varphi$  can be written in the form $d\varphi= G^\ast df$ for some Hamiltonian $f$. Therefore the {\it Hamiltonian Stationary Equation} is
equivalent to the {\it $S^1$-harmonic map equation} with respect to the $G^\ast g_{{\R}^4}$ metric on $\Sigma$ :
\be
\label{I.23}
d^\ast \lf({\frak g}^{-1} d{\frak g}\rg)=0\ ,
\ee
where again the Hodge operation $\ast$ is the one given by the complex structure on $\Sigma$ induced by $G^\ast g_{{\R}^4}$. Observe that  $\frak g=e^{i\,\theta_0}$ is a constant $S^1-$harmonic map if and only if 
$G$ is minimal and Lagrangian (i.e. $\vec{H}_G=0$), which is also equivalent to the fact that $G$ is {\it calibrated} by the form 
\be
\label{I.24}
\Om:=e^{-i\,\theta_0} dz_1\wedge dz_2
\ee
and $G$ realizes a so called {\it special Lagrangian immersion}.

All the above up to (\ref{I.23}) extends  to the general case of a K\"ahler-Einstein Surface and up to (\ref{I.24}) to the general case of a K\"ahler-Einstein Surface with trivial {\it canonical bundle}: for which there exists
a global nowhere vanishing holomorphic  $(2,0)$ form (Calabi-Yau surfaces).

In their pioneering work on the analysis of {\it Hamiltonian Stationary Maps} R. Schoen and J. Wolfson aimed at constructing in any {\it integral Lagrangian homology class} (or also in the slightly more restrictive constraint of any Hamiltonian isotopy class \cite{Oh0}) of a closed K\"ahler-Einstein manifold a Lagrangian surface
minimizing the area in this class. The main contribution of \cite{SW} is the successful implementation of the minimization procedure. They proved that any such a class is realized  by a smooth minimal immersion away from isolated singular  points and that (\ref{I.23}) holds. The singularities however can be ``worse'' than classical branched points which are common for minimal surfaces and which are also present of course. For instance, at such singularities, the Gauss map of the immersion cannot be extended smoothly and it is proved in \cite{SW} that these singularities (called ``Schoen-Wolfson cones'') correspond to singularities around which the $S^1$-harmonic map has non-zero degree (in fact $+1$ or $-1$ degree) and coincide exactly with the singularities of ${\frak g}$.
In \cite{Wo}, Wolfson is giving examples of Lagrangian spherical integer homology classes whose $S^2$ Lagrangian area minimizers are not minimal and hence the associated $S^1-$harmonic maps ${\frak g}$ must have singularties and are only in $W^{1,(2,\infty)}(S^2, S^1)$ (and not in $W^{1,2}(S^2,S^1)$). 

\subsubsection{The Renormalized Lagrangian Willmore Energy}
From (\ref{I.20}) we deduce that the {\it Willmore Energy} for  a Lagrangian immersion $G$ into a K\"ahler-Einstein surface  is  given by
\be
\label{I.25}
W(G):=\frac{1}{4}\int_\Sigma |d {\frak g}|^2_{g_G}\ dvol_G\ .
\ee
This is just $1/2$ of the Dirichlet energy of the {\it Lagrangian angle function} $\frak g$.

From the previous subsection it is clear that for any {\it Hamiltonian stationary surface} in a Calabi-Yau 2-fold which is not {\it special Lagrangian} the {\it Willmore energy} is infinite
\[
W(G)=+\infty
\]
and there is an obvious need to renormalize it.

 Following the first part of the paper we introduce the {\it Renormalized Lagrangian-Willmore Energy}. Let $G$ be a Lagrangian map from a closed surface $\Sigma$ into a K\"ahler-Einstein surface $N$ realizing a Lipschitz weak immersion (in the sense introduced in \cite{Riv}) away from point singularities with an underlying smooth conformal structure $h$ and in such a way that $\frak g$ is in $W^{1,p}(\Sigma,S^1)$ for some $p>1$ and has finitely many topological singularities. We call such an immersion a {\it singular immersion}.
Let $a_{\frak g}$, $b_{\frak g}$ and ${\frak h}_{\frak g}$ be such that
\be
\label{I.26}
-i{\frak g}^{-1}\, d{\frak g}=\ast da_{\frak g}+db_{\frak g}+{\frak h}_{\frak g},
\ee
where ${\frak h}_{\frak g}$ realizes an harmonic one form.  Following the first part of paper we introduce the {\it Renormalized Lagrangian-Willmore Energy}.
\be
\label{I.27}
{\mathcal W}(G):=\frac{1}{4}\int_\Sigma |du_{\frak g}|^2_h+|db_{\frak g}|_h^2+|{\frak h}_{\frak g}|^2_h\ dvol_h\ ,
\ee
 where $u_{\frak g}$ is the ``$S^2$ lift'' of $\frak g$ introduced in Definition~\ref{df-I.1}. Such a singular Lagrangian immersion being given, we define the {\it  singular Lagrangian degree} of $G$ to be 
 \[
 \mbox{deg}_{Lag}(G)=\mbox{deg}(u_{\frak g})\ .
 \]
 From the previous we deduce the following Proposition. 
 \begin{Prop}
 \label{pr-V.1}
 Let $G$ be a singular Lagrangian immersion of $S^2$ into a K\"ahler-Einstein manifold. The map $G$ is Hamiltonian stationary if and only if $b_{\frak g}\equiv0$ and we have then
 \[
 {\mathcal W}(G)=2\pi\, \operatorname{deg}_{Lag}(G)\ .
 \]
 \hfill $\Box$
 \end{Prop}
 We propose the following open problems.
 \begin{oq}
 Study the sequential weak closure of singular weak Lagrangian immersions under Renormalized Lagrangian-Willmore Energy control in the spirit of \cite{Riv}.
 \end{oq}
 Observe that thanks to inequality (\ref{I.18}) (or more exactly its counterpart in the Lagrangian immersion framework) the control of the point singularities is guaranteed by the control of the {\it Renormalized Lagrangian-Willmore Energy}.
 \begin{oq}
 Study the minimization of the Renormalized Lagrangian-Willmore Energy among  singular weak Lagrangian immersions in Hamiltonian isotopy classes. 
 \end{oq}
 
\subsection{Frame Energy}
Let $\vec{\Phi}$ be an immersion of a torus $\Sigma=T^2$ in ${\R}^n$ and let ${\frak e}:=(\vec{e}_1,\vec{e}_2)$ be an associated tangent frame, that is ${\frak e}(x)$ forms an orthonormal basis of
$\vec{\Phi}_\ast T_x\Sigma$. The associated frame energy is just the Dirichlet energy of ${\frak e}$ :
\be
\label{I.28}
F(\vec{\Phi},{\frak e}):=\frac{1}{2}\int_{T^2}|d\vec{e}_1|^2_{g_{\vec{\Phi}}}+|d\vec{e}_2|^2_{g_{\vec{\Phi}}}\ dvol_{g_{\vec{\Phi}}}.
\ee
It has been originally introduced in \cite{Top}. Its minimization in isotopy classes of immersions has been performed in \cite{MR} while its link with the Alvarez-Polyakov anomaly has been established in \cite{Riv1}.
For a fixed immersion the optimal frame satisfies the {\it Coulomb condition}
\be
\label{I.29}
d^\ast(\vec{e}_1\cdot d\vec{e}_2)=0.
\ee
Observe that the passage from one such a frame to another {\it Coulomb frame} $(\vec{f}_1,\vec{f}_2)$ is achieved through an $S^1$-harmonic map $g$ given by
\be
\label{I.30}
\vec{e}_1\cdot d\vec{e}_2=\vec{f}_1\cdot d\vec{f}_2-ig^{-1}\,dg\ .
\ee
It is clear that for a closed oriented surface $\Sigma$ of genus different from $1$ any such a frame has singularities and its Dirichlet energy is infinite
\[
F(\vec{\Phi},{\frak e})=+\infty.
\]
Thus there is an obvious need to renormalize the frame energy. 

The renormalization  we are proposing corresponds to a ``bundle version'' of the passage from $E$ to ${\mathcal E}$ in the Introduction. Let $\vec{e}_1$ be a (possibly singular at isolated points)
section of the unit tangent bundle to the immersion $\vec{\Phi}$ of a simply connected\footnote{We are avoiding harmonic forms in this first presentation.} closed oriented smooth surface $\Sigma$.
Assume that
\begin{align*}
\vec{e}_2\cdot d\vec{e}_1\in L^p(\Sigma)
\end{align*}
for some $p>1$.
Let $a_\mathfrak{e},b_\mathfrak{e}\in W^{1,p}(\Sigma)$ such that
\begin{align}\label{Equation: Hodge decomposition for frames}
\begin{cases}
\vec{e}_2\cdot d\vec{e}_1=\ast d a_\mathfrak{e}+db_\mathfrak{e}\text{ in }\Sigma\\[5mm]
\int_{\Sigma}a_\mathfrak{e}=0\\[5mm]
\int_{\Sigma}b_\mathfrak{e}=0.
\end{cases}
\end{align}
Let
\begin{align*}
v_\mathfrak{e}:= e^{a_\mathfrak{e}}\vec{e}_1\text{ on }D^2.
\end{align*}
Then $v_\mathfrak{e}$ defines a section of the tangent bundle of $\Sigma$.
We remark that the section $v_\mathfrak{e}$ has the following representation in charts: let
\begin{align*}
\phi: \Omega\to U
\end{align*}
be a conformal diffeomorphism from a domain $\Omega\subset\mathbb{C}$ to an open subset $U\subset\Sigma$. For $i\in \{1,2\}$ let
$$e_i:=\phi^\ast \vec{e}_i.$$
Then, in the standard metric of $\Omega$,
$$\lvert e_1\rvert=e^{-\lambda},$$
where $\lambda$ denotes the conformal factor of the map $\phi$.
Therefore the Hodge decomposition in (\ref{Equation: Hodge decomposition for frames}) takes the form
\begin{align*}
e^{\lambda}e_2\cdot d(e^{\lambda}e_1)=e^{2\lambda}e_2\cdot d e_1=\ast d(a_\mathfrak{e}-\lambda)+d b_\mathfrak{e}.
\end{align*}
Here we can see how $e^{\lambda}e_1$ corresponds heuristically to the function $g$ studied above.\\
We define the renormalized frame energy of the frame $\mathfrak{e}$ as follows:
\begin{align}\label{Equation: Definition of renormalized frame energy}
\mathcal{F}(\mathfrak{e}):=\int_{\Sigma}\frac{e^{2a_\mathfrak{e}}}{(1+e^{2a_\mathfrak{e}})^2}\lvert D v_\mathfrak{e}\rvert_{g_{\vec{\Phi}}}^2+\frac{1}{4}\int_{\Sigma}\lvert db_\mathfrak{e}\rvert_{g_{\vec{\Phi}}}^2.
\end{align}
Here $D$ denotes the covariant derivative of $\Sigma$.\\
Notice that the weight in the first integral of (\ref{Equation: Definition of renormalized frame energy}) corresponds to the one generated by the differential of an inverse stereographic projection on the tangent planes.\\
Notice also that the energy $\mathcal{F}$ depends on the immersion $\vec{\Phi}$ only through the metric on $\Sigma$ induced by $\vec{\Phi}$. In particular the energy $\mathcal{F}$ is an intrinsic object.\\
In the following we assume that the frame $\mathfrak{e}$ satisfies the equation:
\begin{align*}
\Delta_g a_\mathfrak{e}=2\pi\sum_{i=1}^Qd_i\delta_{p_i},
\end{align*}
where $Q\in \mathbb{N}$ and for any $i\in \{1,...,Q\}$ $d_i\in \mathbb{Z}\smallsetminus\{0\}$, $p_i\in \Sigma$. We call the points $p_1,...,p_Q$ \textit{topological singularities} of $\mathfrak{e}$.\\
With the methods presented above one proves the following Lemma.
\begin{lem}
Let $\Sigma$ and $\mathfrak{e}$ be as above, assume that
\begin{align*}
\mathcal{F}(\mathfrak{e})<\infty.
\end{align*}
Then the following conditions are equivalent
\begin{enumerate}
\item $\mathfrak{e}$ is a critical point of the energy $\mathcal{E}$, in the sense that
\begin{align*}
\forall \psi\in C^\infty(\Sigma,\mathbb{R})\quad \frac{d}{dt}\bigg\vert_{t=0}\mathcal{E}(e^{it\psi}\mathfrak{e})=0,
\end{align*}
\item $\mathfrak{e}$ satisfies the Coulomb condition, i.e. $b_\mathfrak{e}\equiv 0$,
\item the section $v_\mathfrak{e}$ is holomorphic outside of the topological singularities of $\mathfrak{e}$.
\end{enumerate}
\end{lem}
If $\Sigma$ is not simply connected there is a unique harmonic 1-form $h_{\frak e}$ such that
\[
\vec{e}_2\cdot d\vec{e_1}=\ast da_{\frak e} +db_{\frak e}+h_{\frak e}\
\]
and the renormalized frame energy becomes
\begin{align}\label{Equation: Renormalized energy for generalized manifolds}
\mathcal{F}(\mathfrak{e}):=\int_{\Sigma}\frac{e^{2a_\mathfrak{e}}}{(1+e^{2a_\mathfrak{e}})^2}\lvert D v_\mathfrak{e}\rvert_{g_{\vec{\Phi}}}^2+\frac{1}{4}\int_{\Sigma}\left(\lvert db_\mathfrak{e}\rvert_{g_{\vec{\Phi}}}^2+\lvert h_\mathfrak{e}\rvert_{g_{\vec{\Phi}}}^2\right).
\end{align}
If moreover $\Sigma$ has a boundary, one could define a renormalized frame energy in the same spirit as (\ref{Equation: Renormalized energy for generalized manifolds}) by means of an $L^p$-Hodge decomposition for manifolds with boundary (see Corollary 10.5.1 in \cite{Iwaniec}).\\
It would be interesting to study the following question.
\begin{oq}
Consider an immersion $\vec{\Phi}$ of a an oriented surface with boundary and study the link between the minimal  renormalized frame energy among all tangent frames such that the first vector is tangent to the boundary
and the Alvarez-Polyakov anomaly associated to this immersion.\hfill $\Box$
\end{oq}
\section{Appendix: some auxiliary results}
\reset

\begin{lem}\label{Lemma: Hodge decomposition in L2,1}
Let $p\in (1,\infty)$, $q\in [1,\infty]$, let $X\in L^{p,q}(D^2)$.
Then there exists $\eta\in W^{1, (p,q)}(D^2)$, $\xi\in W^{1, (p,q)}_0(D^2)$ so that
$$X=\nabla \eta+\nabla^\perp \xi\text{ in }D^2$$
and
$$\lVert \eta\rVert_{W^{1, (p,q)}}\leq C\lVert X\rVert_{L^{p,q}},\quad\lVert \xi\rVert_{W^{1, (p,q)}}\leq C\lVert X\rVert_{L^{p,q}}$$
for some constant $C$ independent from $X$.
\end{lem}
\begin{proof}
The result is well known for $L^p$ spaces. In particular, for fixed $p\in (1,\infty)$ there exists a continuous linear map from $L^p(D^2)$ to $W^{1,p}(D^2)$ sending a vector space $X$ to the solution $\eta$ of
$$\begin{cases}
\Delta \eta=\text{div}X&\text{ in }D^2\\[5mm]
\eta=0&\text{ on }\partial D^2.
\end{cases}$$
By interpolation, one can define an analogous map from $L^{p,q}(D^2)$ to $W^{1, (p,q)}(D^2)$.
Now given $X\in L^{p,q}(D^2)$ and $\eta$ as above, there holds
$$\text{div} (X-\nabla \eta)=0.$$
Therefore there exists $\xi\in W^{1, (p,q)}(D^2)$ such that $\nabla^\perp\xi=X-\nabla \eta$.
\end{proof}

\begin{lem}\label{Lemma: continuous extension from W11 to W 12infty}
Let $g_0\in W^{1,1}(\partial D^2, S^1)$. Then there exists an harmonic map $g\in W^{1, (2,\infty)}(D^2)$ with trace $g_0$ on $\partial D^2$, and the map sending $g_0\in W^{1,1}(\partial D^2, S^1)$ to $g\in W^{1, (2,\infty)}(D^2, S^1)$ is a bounded continuous map.
\end{lem}
\begin{proof}
Let $g_0\in W^{1,1}(\partial D^2, S^1)$.
First we show that $g_0$ has a lift in $BV(\partial D^2, \mathbb{R})$ (up to an additive constant). In fact, for any $\theta\in [0, 2\pi]$ let
$$\phi(\theta)=-i\int_0^\theta g_0^{-1}(e^{i\alpha})\partial_\theta g_0(e^{i\alpha})d\alpha.$$
Set
$$\lambda(e^{i\theta})=\phi(\theta)$$
for any $\theta\in [0, 2\pi)$.
Then $\lambda\in BV(\partial D^2, \mathbb{R})$ with
\begin{align}\label{Equation: Estimate BV W11}
\lVert \lambda\rVert_{BV}\leq C\lVert g_0\rVert_{W^{1,1}}
\end{align}
and
$$g_0=g_0(1)e^{i\lambda}\text{ on }\partial D^2.$$
We claim that the map
\begin{align*}
W^{1,1}(\partial D^2, S^1)\to BV(\partial D^2, \mathbb{R}),\quad g_0\mapsto \lambda
\end{align*}
is continuous. In fact let $g_0\in W^{1,1}(\partial D^2, S^1)$ and let $(g_0^n)_{n\in \mathbb{N}}$ be a sequence in $W^{1,1}(\partial D^2, S^1)$ such that
$$g_0^n\to g_0\text{ in }W^{1,1}(\partial D^2).$$
Observe that by estimate (\ref{Equation: Estimate BV W11}) for any $n\in \mathbb{N}$
\begin{align*}
\lVert \lambda_{g_0^n}-\lambda_{g_0}\rVert_{BV}\leq C\lVert g_0^n g_0^{-1}\rVert_{W^{1,1}}.
\end{align*}
Now let $\Lambda\subset \mathbb{N}$ the index set of a subsequence of $(g_0^n)_{n\in \mathbb{N}}$ such that
\begin{align*}
g_0^n\to g_0\text{ a.e. along }\Lambda.
\end{align*}
Then by Dominated Convergence
\begin{align*}
\lim_{\substack{n\to\infty\\ n\in \Lambda}}\lVert g_0^n g_0^{-1}\rVert_{L^1}=0.
\end{align*}
Moreover for any $n\in \mathbb{N}$
\begin{align*}
\partial_\theta(g_0^ng_0^{-1})=\partial_{\theta}(g_0^n-g_0)g_0^{-1}+(g_0^n-g_0)\partial_\theta g_0^{-1},
\end{align*}
therefore
\begin{align}\label{Equation: Estimate in two parts for the W11 norm of difference}
\lVert \partial_\theta(g_0^ng_0^{-1})\rVert_{L^1}\leq \lVert \partial_\theta(g_0^n-g_0)\rVert_{L^1}+\lVert \lvert \partial_\theta g_0\rvert\lvert g_0^n-g_0\rvert\rVert_{L^1}
\end{align}
and by Dominated Convergence the right hand side of (\ref{Equation: Estimate in two parts for the W11 norm of difference}) converges to zero along $\Lambda$.
Thus for any subsequence of $(\lambda_{g_0^n})_{n\in \mathbb{N}}$ there exists a further subsequence along which
$$\lambda_{g_0^n}\to\lambda_{g_0}\text{ in }BV(\partial D^2).$$
This concludes the proof of the Claim.\\
Next we consider the following Cauchy problem:
\begin{align}\label{Equation: Cauchy problem for psi in BV}
\begin{cases}
\Delta \psi=0&\text{ in }D^2\\[5mm]
\psi=\lambda&\text{ on }\partial D^2.
\end{cases}
\end{align}
We claim that the solution $\psi$ to (\ref{Equation: Cauchy problem for psi in BV}) lies in $W^{1, (2,\infty)}(D^2)$.
To see this, assume first that $\lambda$ is a smooth function.
Now let $X$ be a vector field on $\overline{D^2}$ smooth up to the boundary, let $\eta\in W^{1, (2,1)}_0(D^2)$ and $\xi\in W^{1, (2,1)}(D^2)$ as in Lemma \ref{Lemma: Hodge decomposition in L2,1}, so that
$$ X=\nabla \eta+\nabla^\perp \xi$$
and
$$\lVert \xi\rVert_{W^{1, (2,1)}}\leq C\lVert X\rVert_{L^{2,1}}.$$
Then
\begin{align}\label{Equation: duality to show harmonic function in W121}
\int_{D^2} X \nabla\psi=\int_{D^2}\nabla \eta\nabla\psi+\int_{D^2}\nabla^\perp \xi\nabla\psi=\int_{D^2}\nabla^\perp \xi\nabla\psi=\int_{\partial D^2}\xi \partial_\theta\psi=\int_{\partial D^2}\xi \partial_\theta \lambda.
\end{align}
Here we used the fact that since $\eta$ has vanishing trace on $\partial D^2$ and $\psi$ is harmonic, 
$$\int_{D^2}\nabla \eta \nabla \psi=0.$$
We also used the fact that, by the divergence Theorem,
$$\int_{D^2}\nabla^\perp \xi\nabla\psi=-\int_{D^2}\text{div}(\xi\nabla^\perp\psi)=-\int_{\partial D^2}\xi\nabla^{\perp}\psi\cdot \nu=\int_{\partial D^2} \xi \partial_\theta \psi.$$
Since $W^{1, (2,1)}(D^2)$ embeds continuously in $L^\infty(D^2)$, it follows from (\ref{Equation: duality to show harmonic function in W121}) that
$$\lVert \nabla \psi\rVert_{L^{2,\infty}}\leq C \lVert \lambda\rVert_{BV}$$
and therefore\footnote{This can be shown by contradiction, just as in the classical proof of Poincar\'e inequality.}
\begin{align}\label{Equation: Estimate of psi through BV}
\lVert \psi\rVert_{W^{1,(2,\infty)}}\leq C\lVert \lambda\rVert_{BV}.
\end{align}
By density of the smooth functions in $BV(\partial D^2)$ the estimate extends by continuity to any $\lambda\in BV(\partial D^2)$.\\
Now given $g_0\in W^1(\partial D^2, S^1)$ let $\lambda_{g_0}$ and $\psi_{\lambda_{g_0}}$ as above.
Set
\begin{align*}
g:=g_0(1)e^{i\psi_{\lambda_{g_0}}}\text{ on }\partial D^2.
\end{align*}
Then by construction $g\vert_{\partial D^2}=g_0$,
\begin{align*}
\lVert \nabla g\rVert_{L^{2,\infty}}\leq C\lVert g_0\rVert_{W^{1,1}}
\end{align*}
and the prescription
$$g_0\mapsto g$$
defines a continuous map
\begin{align*}
W^{1,1}(\partial D^2, S^1)\to W^{1,(2,\infty)}(D^2, S^1).
\end{align*}
\end{proof}

\begin{lem}\label{Lemma: Pseudo Courant Lebesgue}
Let $R\in (0,1)$, let $(u_n)_{n\in \mathbb{N}}$ be a sequence of non-negative functions in $W^{1,1}(D^2\smallsetminus D^2_R)$ defined pointwise. Assume that
\begin{align*}
\sup_{n\in\mathbb{N}}\int_{D^2\smallsetminus D_R^2}u_n\leq C
\end{align*}
for some constant $C$.\\
Let $\varepsilon>0$. Then there exists a measurable subset $A\subset (R,1)$ with positive Lebesgue measure and a subsequence, say indexed by $\Lambda\subset \mathbb{N}$, such that
\begin{align*}
\int_{\partial D^2_R}u_n\leq C_\varepsilon:=\frac{C}{1-R}+\varepsilon
\end{align*}
for any $R\in A$, for any $n\in \Lambda$.
\end{lem}
\begin{proof}
Assume by contradiction that the statement is false.
Then for a.e. $r\in (R,1)$ there exists $N(r)$ such that for any $n\geq N(r)$
\begin{align*}
\int_{\partial D^2_r}u_n> C_\varepsilon
\end{align*}
For any $n\in \mathbb{N}$ set
\begin{align*}
A(n)=\{r\in (R,1)\text{ s.t. } N(r)\leq n\}.
\end{align*}
Then $A(n)\subset A(m)$ for any $n,m\in \mathbb{N}$ with $n\leq m$ and
\begin{align*}
\left\lvert \bigcup_{n\in \mathbb{N}} A(n)\right\rvert=1-R.
\end{align*}
Observe that for any $n\in \mathbb{N}$
\begin{align*}
C\geq \int_{A(n)}\left(\int_{\partial D^2_r}u_n\right)dr\geq \lvert A(n)\rvert C_\varepsilon.
\end{align*}
If we let $n$ tend to infinity we obtain
\begin{align*}
C\geq C_\varepsilon(1-R),
\end{align*}
a contradiction.
\end{proof}


\begin{lem}\label{Lemma: H12 convergence for composition with Lipschitz}
Let $(\varphi_n)_{n\in \mathbb{N}}$ be a sequence in $H^\frac{1}{2}(\partial D^2)$ converging in $H^\frac{1}{2}(\partial D^2)$ to a function $\varphi$, i.e.
\begin{align*}
\varphi_n\to \varphi\text{ in }H^\frac{1}{2}(\partial D^2).
\end{align*}
Let
$$
F: \mathbb{R}\to\mathbb{R}^n
$$
be a Lipschitz-continuous function. Then
\begin{align*}
F\circ\varphi_n\to F\circ\varphi\text{ in }H^\frac{1}{2}(\partial D^2, \mathbb{R}^n).
\end{align*}
\end{lem}
\begin{proof}
It is clear that
\begin{align*}
F\circ\varphi_n\to F\circ\varphi\text{ in }L^2(\partial D^2, \mathbb{R}^n).
\end{align*}
We would like to show that
\begin{align}\label{Equation: Desired H12 convergence}
\lim_{n\to\infty}\int_{\partial D^2}\int_{\partial D^2}\frac{\lvert (F(\varphi_n(x))-F(\varphi(x)))-(F(\varphi_n(y))-F(\varphi(y)))\rvert^2}{\lvert x-y\rvert^2}dxdy.
\end{align}
Observe that for a.e. $(x,y)\in \partial D^2\times \partial D^2$.
\begin{align*}
\left\lvert (F(\varphi_n(x))-F(\varphi(x)))-(F(\varphi_n(y))-F(\varphi(y)))\right\rvert\leq L\left(\lvert \varphi_n(x)-\varphi_n(y)\rvert+\lvert \varphi(x)-\varphi(y)\rvert\right),
\end{align*}
where $L$ denotes the Lipschitz constant of the function $F$.\\
Now since
\begin{align*}
\varphi_n\to \varphi\text{ in }H^\frac{1}{2}(\partial D^2),
\end{align*}
there holds
\begin{align*}
\left\lvert\frac{\varphi_n(x)-\varphi_n(y)}{\lvert x-y\rvert}\right\rvert^2\to \left\lvert\frac{\varphi(x)-\varphi(y)}{\lvert x-y\rvert}\right\rvert^2\text{ in }L^1(\partial D^2\times\partial D^2).
\end{align*}
Therefore there exists a function
$$B\in L^1(\partial D^2\times\partial D^2)$$
and a subsequence of $(\varphi_n)_{n\in \mathbb{N}}$, say indexed by $\Lambda\subset\mathbb{N}$, so that for any $n\in \Lambda$
\begin{align*}
\left\lvert\frac{\varphi_n(x)-\varphi_n(y)}{\lvert x-y\rvert}\right\rvert^2\leq B(x,y)\text{ a.e. in }\partial D^2\times\partial D^2.
\end{align*}
Thus by Dominated Convergence (\ref{Equation: Desired H12 convergence}) holds true for a subsequence. Since this argument can be repeated for any subsequence of $(\varphi_n)_{n\in \mathbb{N}}$, the statement holds true.
\end{proof}

\end{document}